\documentclass[12pt]{article}
\usepackage{graphicx, times}
\usepackage{amsfonts}
\usepackage{amsmath}
\usepackage{amsthm}
\usepackage{amssymb}
\usepackage{fancyhdr}
\usepackage{indentfirst}
\usepackage{titlesec}
\usepackage{hyperref}
\usepackage{abstract}

\newcommand{\al}{\alpha}
\newcommand{\be}{\beta}

\newcommand{\R}{\mathbb{R}}
\newcommand{\N}{\mathbb{N}}

\newcommand{\Si}{\Sigma}
\newcommand{\de}{\delta}

\newcommand{\pr}{\prime}
\newcommand{\Om}{\Omega}

\newcommand{\Ga}{\Gamma}
\newcommand{\ga}{\gamma}

\newcommand{\ti}{\tilde}

\newcommand{\Diff}{\textrm{Diff}}
\newcommand{\Met}{\textrm{Met}}
\newcommand{\PSL}{\textrm{PSL}}

\newcommand{\Rom}[1]{\uppercase\expandafter{\romannumeral #1}}
\newcommand{\rom}[1]{\expandafter\romannumeral #1}

\theoremstyle{definition}
\newtheorem{theorem}{Theorem}[section]
\newtheorem{proposition}[theorem]{Proposition}
\newtheorem{definition}[theorem]{Definition}

\newtheorem{corollary}[theorem]{Corollary}
\newtheorem{example}[theorem]{Example}

\theoremstyle{plain}
\newtheorem{lemma}[theorem]{Lemma}

\theoremstyle{remark}
\newtheorem{remark}[theorem]{Remark}

\numberwithin{equation}{section}

\begin{document}

\titleformat{\section}
   {\normalfont\bfseries\large\filcenter}
   {\arabic{section}}
   {12pt}{}
\titleformat{\subsection}
   {\normalfont\bfseries}
   {\arabic{section}.\arabic{subsection}}
   {11pt}{}
%
%
%
%

\pagestyle{headings}
\renewcommand{\headrulewidth}{0.4pt}

\title{\textbf{On the existence of min-max minimal surface of genus $g\geq 2$}}
\author{Xin Zhou\\}
\maketitle

\pdfbookmark[0]{On The Existence Of Min-max Minimal Surface of Genus $g\geq 2$}{beg}

\renewcommand{\abstractname}{}    
\renewcommand{\absnamepos}{empty} 
\begin{abstract}
\noindent\textbf{Abstract:} In this paper, we build up a min-max theory for minimal surfaces using sweepouts of surfaces of genus $g\geq 2$.  We develop a direct variational methods similar to the proof of the famous Plateau problem by Douglas \cite{Do} and Rado \cite{Ra}. As a result, we show that the min-max value for the area functional can be achieved by a bubble tree limit (see \cite{Pa}) consisting of branched genus-$g$ minimal surfaces with nodes, and possibly finitely many branched minimal spheres. We also prove a Colding-Minicozzi type strong convergence theorem similar to the classical mountain pass lemma \cite{St}. Our results extend the min-max theory by Colding-Minicozzi and the author to all genera.
\end{abstract}



\section{Introduction}

\subsection{Background}
Existence theory of minimal surfaces originated from the celebrated proof of classical Plateau Problem by Douglas \cite{Do} and Rado \cite{Ra} (see more history in \cite[Chap 4]{CM11}) in 1930s. These minimal surfaces are parametrized by conformal harmonic maps\footnote{See \cite[Lemma 1.4]{SU1}.}. Since then, there are lots of interesting results concerning general existence theory of minimal surfaces using conformal harmonic parametrization\footnote{Another story is the geometric measure theory, and we refer to \cite{Si83} for details.}. Among them, Schoen-Yau \cite{SY} built up an existence theory for incompressible minimal surfaces to study the topology of three manifolds with non-negative scalar curvature. Around the same time, Sacks-Uhlenbeck developed a general existence theory for minimal surfaces in compact manifold using Morse theory for perturbed energy functional \cite{SU1, SU2}. Michallef-Moore used the minimal spheres in \cite{SU1} to prove the topological sphere theorem \cite{MM}. Chen-Tian \cite{CT} gave a general existence theorem for minimal surfaces of arbitrary genus by extending \cite{SU1, SU2} to stratified Riemann surfaces. These results mainly work when the minimal surfaces are area-minimizing in a homotopy class.


Besides the area minimizing case, the min-max theory for minimal surfaces has attracted more interest recently (cf. \cite{J1, CM2, CM3})\footnote{For the geometric measure theory part, see \cite{CD, P81}.}. One remarkable work was given by Colding and Minicozzi in \cite{CM2, CM3}, where they constructed min-max minimal spheres and proved the finite time extinction for three-dimensional Ricci flow under certain topological conditions by studying the evolution of the area of the min-max minimal spheres. A key novelty of their work is a strong convergence result compared to \cite{J1} (see more discussion in \S \ref{further discussion}). Motivated by their work, the author studied the variational construction of min-max minimal tori in \cite{Z}. The difference between spheres and surfaces of genus greater than zero is that the moduli space of conformal structures is nontrivial. The author developed a uniformization result in \cite{Z} to deal with this technical difficulty in the case of tori\footnote{In the case of tori, \cite{DLL} also gave a method to deal with moduli space in an evolutional  setting.}


In the area minimizing case, the study of high genus minimal surfaces achieved many interesting results \cite{SY, SU2, CT}. Therefore a min-max theory for surfaces of arbitrary genus is then a natural question. Using the geometric measure theory setting (see \cite{CD}), Marques and Neves recently \cite{MN} gave an application of the min-max minimal surfaces of arbitrary genus to get certain rigidity results on positive curved compact manifold. Motivated by these works, we build up a min-max theory for minimal surfaces using sweep-outs of genus-$g$ surfaces ($g\geq 2$), hence we extend the results \cite{CM3, Z} to the full generality. 

\subsection{Main result}
To state the main theorem, we recall a few notations here (more detailed versions are given in \S \ref{notations}). Let $\Si_0$ be a Riemann surface of genus $g$ ($g\geq 2$), and $(N, h)$ a closed Riemannian manifold of dimension no less than $3$. Denote $C^0\cap W^{1, 2}(\Si_0, N)$ by the Banach space of mappings $u: \Si_0\rightarrow N$ which are both $C^0$ and $W^{1, 2}$. We call a one-parameter family of mappings $\ga: [0, 1]\rightarrow C^0\cap W^{1, 2}(\Si_0, N)$ a {\em sweep-out}, if
\begin{itemize}
\vspace{-5pt}
\setlength{\itemindent}{1em}
\addtolength{\itemsep}{-0.7em}
\item $\ga(0)$, $\ga(1)$ are mapped to points or a curve;
\item The mapping $\ga$ is homotopically non-trivial in $C^0\cap W^{1, 2}(\Si_0, N)$.
\end{itemize}

\begin{example}
one such example comes from the Heegaard splitting of three manifolds. Let $(M^3, h)$ be an oriented three-manifold, with Heegaard genus $g_0\geq 2$, then there is a smooth foliation $\{\Si_t\}_{t\in[0, 1]}$, where $\Si_0$, $\Si_1$ are graphs (curves), and $\Si_t$ is an embedded genus-$g_0$ surface for $t\in(0, 1)$. Let $\Si_{g_0}$ be a fixed Riemann surface of genus $g_0$, then we can then automatically find a parametrization $\ga: [0, 1]\rightarrow C^2(\Si_{g_0}, M)$, where $u_t=\ga(t)$ maps $\Si_{g_0}$ to $\Si_t$.
\end{example}

The space of sweep-outs is denoted by (see Definition \ref{variational space}),
$$\Omega=\big\{\ga: \gamma(t) \textrm{ is continuous as a map } [0, 1]\rightarrow C^{0}\cap W^{1,2}(\Sigma_{0}, N)\big\}.$$
Now we can formulate a min-max theory using sweep-outs in $\Om$. Given a homotopy class $[\be]$ in $\Om$, the min-max value, called {\em width}\footnote{See \cite[4.1(3)]{P81}\cite[\S 1.1]{CD} for similar definitions in the geometric measure theory setting.} (see Definition \ref{W and W-E}), is defined by
$$\mathcal{W}=\underset{\rho \in[\be]}{\inf}\underset{t\in[0,1]}{\max}Area\big(\rho(t)\big),$$
where $Area$ is the area functional defined by:
$$Area(u)=\int_{\Si_0}du^{*}(dvol_h)\footnote{$dvol_h$ is the volume form of $(N, h)$.},\quad \textrm{ for } u\in W^{1, 2}(\Si_0, N).$$
We will also use the {\em harmonic energy functional} $E$\footnote{For more other equivalent definitions and properties of $Area$ and $E$, we refer to \cite{J1, SU1, CM3}.}. Let $\al$ be a Riemannian metric on $\Si_0$, then $E$ is define as
$$E(u)=\frac{1}{2}\int_{\Si_0}\|du\|_{\al, h}^2 dvol_\al.$$
$E$ depends only on $u$ and the conformal class of $\al$. Critical point of $E$ is called {\em harmonic map}. Denote $\mathcal{T}_g$ by the Teichm\"uller space on Riemann surface of genus $g$ (see \S \ref{Teichmuller space}.1). It is equivalent to the space of all conformal structures on $\Si_0$ module out the action of isotopy group of $\Si_0$.

Now we can summarize our main theorem as:

\begin{theorem}\label{main theorem}
For any homotopically nontrivial $\beta\in\Omega$, if $\mathcal{W}>0$, there exists a sequence $(\rho_{n}, \tau_{n})$, $\rho_n\in[\beta]$, $\tau_n\in\mathcal{T}_g$, with $\underset{t\in[0,1]}{max}E\big(\rho_{n}(t), \tau_{n}(t)\big)\rightarrow\mathcal{W}$, and for any $\epsilon>0$, there exists a large number $N>0$ and $\delta>0$, such that if $n>N$, then for any $t\in(0, 1)$ satisfying:
\begin{equation}\label{energy condition}
E\big(\rho_{n}(t), \tau_{n}(t)\big)>\mathcal{W}-\delta,
\end{equation}
there are a conformal harmonic map $u_{0}:\overline{\Sigma}_{g}\rightarrow N$ defined on the body $\Sigma^{*}_{g}$ of a genus-$g$ Riemann surface with nodes and possibly finitely many harmonic sphere $u_{i}:S^{2}\rightarrow N$, such that:
\begin{equation}
d_{V}\big(\rho_{n}(t), \underset{i}{\cup}u_{i}\big)\leq\epsilon.
\end{equation}
\end{theorem}
\noindent Here the definition of Riemann surfaces with nodes is given in \S \ref{degeneration of conformal structures}, and $d_{V}$ means varifold distance given in \cite[Appendix A]{CM3}\footnote{See also \cite[\S 2.1(19)]{P81} for another equivalent formulation.}. The theorem follows from the following Theorem \ref{main theorem 2} and the fact that bubble tree convergence (see \S \ref{further discussion}) with energy identity implies varifold convergence \cite[Appendix A]{CM3}.

\subsection{Further discussion}\label{further discussion}

To illustrate the novelty of our result, we need to state a technical version of our main theorem. For that purpose, we need to introduce the notion of {\em bubble tree convergence of harmonic maps}. Bubble tree convergence of harmonic maps originated from the seminal work of Sacks and Uhlenbeck \cite{SU1, SU2}, where they study the existence of harmonic maps in an arbitrary Riemannian manifold. It was then used a lot in geometric analysis \cite{SiY, MM, QT, Pa, CT} and symplectic geometry \cite{Gr, H, PW}. Roughly speaking, given a sequence of harmonic maps from $\Si_0$ to $(N, h)$ with bounded energy, it will automatically converge (up to a subsequence) to a limiting harmonic map on $\Si_0$ away from finitely many energy concentration points. If we rescale the domain near those points, the blow-up sequence will converge to a harmonic map defining on the sphere. Such process can be iterated and will terminate after finitely many steps. The limit will be a tree of harmonic maps. We refer to \cite{Pa} and \cite[Appendix A]{CM} and the proof of Theorem \ref{convergence} for more detailed description of bubble tree convergence.

We also need to use the notion of hyperbolic representation of Teichm\"uler spaces $\mathcal{T}_g$. Denote a triple $(\Si, h, j)$ by a Riemann surface $\Si$ with genus $g\geq 2$, together with a hyperbolic metric $h$ and a compatible complex structure $j$. $\mathcal{T}_g$ can be represented as the space of all such triples $(\Si, h, j)$ module out the isotopic isomorphism group (see \S \ref{degeneration of conformal structures} for more detailed description).

An equivalent version of our main result can be stated as follows: Let $\{\rho_n(t), \tau_n(t)\}$ be as in Theorem \ref{main theorem},
\begin{theorem}\label{main theorem 2}
 For all sequences $\{t_n: t_n\in (0, 1)\}_{n\in\N}$, with $\lim_{n\rightarrow\infty}E\big(\rho_{n}(t_{n}), \tau_{n}(t_{n})\big)=\mathcal{W}$, $\{\rho_n(t_n), \tau_n(t_n)\}$ will converge in the following way:
\begin{itemize}
\vspace{-5pt}
\addtolength{\itemsep}{-0.7em}

\item There exists a sequence $(\Sigma_{n}, h_{n}, j_{n})\in \tau_n(t_n)$, which converge to a hyperbolic Riemann surface with nodes $(\Sigma^{*}_{\infty}, h_{\infty}, j_{\infty})$ (see Definition \ref{convergence of Riemann surfaces}). Let $\overline{\Sigma}_{\infty}$ be the one point compactification of $\Sigma_{\infty}$, then there exist a conformal harmonic map $u_{0}:\big(\overline{\Sigma}_{\infty}, j_{\infty}\big)\rightarrow N$ and some harmonic spheres $\{u_{i}:S^{2}\rightarrow N|\ i=1, \cdots, l\}$, such that $\big(\rho_{n}(t_{n}), (\Sigma_{n}, h_{n}, j_{n})\big)$ bubble converge to a tree $\big(u_{0}, u_{1}, \ldots, u_{l}\big)$, with energy identity:
\begin{equation}\label{energy identity 0}
\underset{n\rightarrow\infty}{\lim}E\big(\rho_{n}(t_{n}), j_{n}\big)=E(u_{0}, j_{\infty})+\underset{i}{\sum}E(u_{i}).
\end{equation}
\end{itemize}

\end{theorem}

The novelty of the main theorem lies on two folds.  First, our result corresponds to a strong mountain pass type lemma in the non-linear analysis \cite[Chap \Rom{2}]{St}. Roughly speaking, in our min-max theory, we find an approximates sequence of sweep-outs $\{\rho_n: [0, 1]\times(\Si_0, \tau_n)\rightarrow N\}_{n\in\N}$, such that every min-max sequence, i.e. $\{(\rho_n(t_n), \tau_n(t_n))\}$ with $\lim_{n\rightarrow\infty}E\big(\rho_{n}(t_n), \tau_{n}(t_n)\big)$ $=\mathcal{W}$, will sub-converge to a bubble tree of branched minimal surfaces. This is a special feature compared to all other versions of min-max theory \cite{P81, CD, J1}, where they can only show the convergence for some special min-max sequence.

The second novelty lies on the energy identity (\ref{energy identity 0}). The possible loss of energy during the bubble tree convergence has attracted a lot of interests during the past thirty years. The energy identity, equivalent to no loss of energy, has played an important role in the study of geometric analysis \cite{J1, Pa, QT, CT}, complex geometry \cite{SiY} and symplectic geometry \cite{PW}. These known results either only deal with the minimizing case \cite{SiY, CT}, or assume some other technical conditions \cite{J1, Pa, PW, QT}. Especially, for bubble tree convergence of harmonic maps defined on $\{(\Si_0, j_n)\}$ with varying conformal structures $\{j_n\}$, \cite{Pa} points out that the energy identity can be false in general. As the second special feature of our result, the energy identity automatically holds during the bubble tree convergence of any min-max sequences defined on surfaces with varying conformal structures.\\

The main difficulty of our theory is due to the complexity of the conformal structures on genus $g\geq 2$ surfaces. We use a variational method analogous to the Plateau Problem. More precisely, we start by taking an arbitrary minimizing sequence of sweep-outs, then we reparametrize to make them almost conformal, and finally we do local perturbation to make them almost compact under the $C^0\cap W^{1, 2}$ topology. The conformal reparametrization uses many features of the Teichm\"uller theory, together with the a priori estimates developed by the author in \cite{Z}. Various representations of the Teichm\"uller space are entangled in the proof. The local perturbation is a delicate adaption of Colding-Minicozzi's local harmonic replacement process \cite[\S 3]{CM3} (see also \cite{Z}), while in our case the possibility of degeneration of conformal structures are much more complicated than \cite{CM3, Z}.


The organization of the paper is as follows. In \S \ref{Sketch of the variational methods}, we review various definitions and properties of Teichm\"uller spaces on a genus $g\geq 2$ surface, and then sketch the variational method. In \S \ref{Conformal parametrization}, we recall the properties of quasi-conformal maps \cite{AB} and quasi-linear quasi-conformal maps \cite[Appendix]{Z}, and prove a strong uniformization result on genus $g\geq 2$ surfaces. In \S \ref{compactification for mappings}, we develop a new version of Colding-Minicozzi's harmonic replacement process \cite[\S 3]{CM3} on genus $g\geq 2$ hyperbolic surfaces. In \S \ref{Convergence results}, we adapt the bubble tree convergence to our setting and finish the whole proof.\\

\noindent\textbf{Acknowledgement:} The author would like to express his gratitude to his advisor Professor Richard Schoen for all of his helpful guidance and constant encouragement. He would like to thank Professor Steven Kerckhoff for teaching him the Teichm\"uller theory. He would also like to thank Professor Gang Tian for his interest in this work.


\section{Sketch of the variational methods}\label{Sketch of the variational methods}

Now let us first recall the approach used by the author in \cite{Z}. In this method, we consider the area functional and energy functional simultaneously. Let $(N, h)$ be the target manifold. Consider the space of sweep-outs $\Omega=\big\{ \gamma(t)\in C^{0}\big([0,1], C^{0}\cap W^{1,2}(T^{2}_{0}, N) \big) \big\}$, where a sweep-out is a one parameter family of mappings $\gamma(t)$ from a torus $T^2_0$ to the target manifold $N$, which satisfy certain degeneration constraints, i.e. $\gamma(0),\ \gamma(1)$ are constant maps or maps to closed curves in $N$. We can define a min-max value $\mathcal{W}=\underset{\rho\in[\beta]}{inf}\underset{t\in[0,1]}{max}Area\big(\rho(t)\big)$ for a homotopy class $[\beta(t)]\subset\Omega$. Suppose that $\mathcal{W}>0$. A natural question is how to find the corresponding critical points. We used classical two dimensional geometric variational methods to find the critical points. First, take an area minimizing sequence of sweep-outs $\tilde{\gamma}_{n}(t)\in\big[\rho\big]$, such that $\lim_{n\rightarrow\infty}\underset{t\in[0,1]}{max}Area\big(\tilde{\gamma}_{n}(t)\big)=\mathcal{W}$. Then we need to change gear to the energy functional $E$. Since energy functional depends not only on the mappings, but also on the conformal structures of the domain, we need to module out the action of conformal group. We consider the following min-max value\footnote{See \cite{Z} for details of the notations.} $\mathcal{W}_{E}=\underset{(\rho, \tau)\in[(\beta, \tau_{0})]}{inf}\quad\underset{t\in[0,1]}{max}E\big(\rho(t), \tau(t)\big)$. In fact, $\mathcal{W}_{E}=\mathcal{W}$ \cite[\S 3]{Z}. In order to module out conformal group action, we need to do reparametrizations on the torus. Let $\tilde{g}_{n}(t)=\tilde{\gamma}_{n}(t)^{*}h$ be the pullback of the ambient metric, which may be degenerate. Using a uniformization result proved in \cite{Z} and a perturbation technique, $\tilde{g}_{n}(t)$ determines a continuous family of elements $\tau_{n}(t)$ in the Teichm\"uller space $\mathcal{T}_{1}$ of torus and a continuous isotopic family of diffeomorphism $h_{n}(t): \big(T^{2}, \tau_{n}(t)\big)\rightarrow\big(T^{2}, \tilde{g}_{n}(t)\big)$, such that if denoting $\gamma_{n}(t)=\tilde{\gamma}_{n}(t)\circ h_{n}(t)$, $\lim_{n\rightarrow}\big[E\big(\gamma_{n}(t), \tau_{n}(t)\big)-Area(\gamma_{n}(t))\big]\rightarrow 0$. After that, we perturb the sequences $\gamma_{n}(t)$ by a modified Colding-Minicozzi's harmonic replacement process \cite[\S 3]{CM3} to a new sequence $\rho_n(t)$ with $\rho_n\in [\ga_n]$, such that $\{\rho_n(t)\}$ satisfy certain compactness property in $C^0\cap W^{1, 2}$ topology. Lastly, we combine the degeneration of conformal structures with the bubble tree convergence to give a combined bubble convergence for the new sequence $\{\rho_n(t): (T^2, \tau_n(t))\rightarrow N\}$ \cite[Theorem 5.1]{Z}. In the limit, we get a bubble tree consisting of a conformal harmonic map from torus together with finitely many harmonic spheres. We also get the energy identity \cite[(45)(46)]{Z}. In fact, we will achieve a strong mountain pass type lemma for $\{\rho_n(t)\}$ \cite[Theorem 1.1]{Z}.

Based on this method, let us describe the approach to high genus cases. 


\subsection{Teichm\"uller spaces of genus $g$ surfaces}\label{Teichmuller space}

Before going into the variational method, let us first review various definitions and properties of the Teichm\"uller spaces $\mathcal{T}_{g}$ and moduli spaces $\mathcal{M}_{g}$ on a genus $g$ surface $\Sigma_{0}$. We will summarize the following facts about $\mathcal{T}_g$ and $\mathcal{M}_g$.
\vspace{-6pt}
\begin{itemize}
\setlength{\itemindent}{1em}
\addtolength{\itemsep}{-0.7em}
\item[$1^{\circ}:$] Definition about Teichm\"uller spaces and Moduli spaces;
\item[$2^{\circ}:$] Marked surface representation of Teichm\"uller spaces;
\item[$3^{\circ}:$] Fuchsian model description for Teichm\"uller spaces;
\item[$4^{\circ}:$] Quasi-conformal maps;
\item[$5^{\circ}:$] Teichm\"uller mappings;
\end{itemize}

$\textbf{1}^{\circ}.$ Denote $\Met_{g}$ by the space of all the Riemannian metrics on a topological surface $\Sigma_0$ of genus $g\geq 2$. Denote $\Diff(\Sigma_0)$ by the orientation-preserving self diffeomorphism groups on $\Sigma_0$, and $\Diff_{0}(\Sigma_0)$ the subgroup of $\Diff(\Sigma_0)$ containing elements isotopic to the identity. Two metrics $ds^{2}$ and $(ds^{2})^{\prime}$ are said to be equivalent in the sense of moduli space, if there exists $w\in \Diff(\Sigma_0)$, such that $w^{*}(ds^{2})^{\prime}$ is conformal to $ds^{2}$. Define all the equivalent classes to be the \emph{moduli space} $\mathcal{M}_g=\Met_{g}/\Diff(\Sigma_0)$. Two metrics $ds^{2}$ and $(ds^{2})^{\prime}$ are said to be equivalent in the sense of Teichm\"uller space, if there exists $w\in \Diff_{0}(\Sigma_0)$, such that $w^{*}(ds^{2})^{\prime}$ is conformal to $ds^{2}$. Define all the equivalent classes to be the \emph{Teichm\"uller space} $\mathcal{T}_{g}=\Met_{g}/\Diff_{0}(\Sigma_0)$. We are also interested in the complex structure of the surfaces. Each $(\Sigma_0, ds^{2})$ automatically has a complex structure compatible with $ds^{2}$ \cite[\S 1.5.1]{IT}. Later on, we will use this complex structure without mentioning it.

\vspace{0.5em}
$\textbf{2}^{\circ}.$ Here we recall the representation of Teichm\"uller spaces by the marked surfaces. We use the description in \cite{IT}. Given a fixed genus $g$-surface $\Sigma_{0}$, consider all the surfaces $(\Sigma, f)$, where $f: \Sigma_{0}\rightarrow\Sigma$ is an orientation-preserving diffeomorphism. We say that $(\Sigma, f)$ and $(\Sigma^{\prime}, g)$ are equivalent in the sense of Teichm\"uller space, if $g\circ f^{-1}: \Sigma\rightarrow\Sigma^{\prime}$ is homotopic to a conformal diffeomorphism from $\Sigma$ to $\Sigma^{\prime}$. We call such a $f$ a \emph{marking}, and $(\Sigma, f)$ a \emph{marked surface}. The set of all equivalent classes of marked surfaces $\big\{\big[(\Sigma, f)\big]\big\}$ is another representation of the Teichm\"uller spaces $\mathcal{T}_{g}$ of genus $g$ \cite[Chap 1]{IT}.

\vspace{0.5em}
$\textbf{3}^{\circ}.$ Let us talk about the Fuchsian model now. By the Uniformization Theorem in complex analysis, all the closed surfaces $\Sigma_{g}$ with genus $g\geq 1$ have their universal covering space the upper half plane $\mathbb{H}$. The covering transformation group of $\pi: \mathbb{H}\rightarrow\Sigma_{g}$ is called \emph{Fuchsian group}, which will be denoted by $\Gamma$, and $(\Sigma_{g}, \Gamma)$ is called \emph{Fuchsian model}. Usually, we also simply call $\Gamma$ a Fuchsian model. In the sense of complex analysis, the holomorphic diffeomorphism group of $\mathbb{H}$ is $\PSL(2, \mathbb{R})$, so $\Gamma$ contains only linear fractional transformations with real coefficients, i,e, $\Gamma\subset \PSL(2, \mathbb{R})$. If we consider the hyperbolic metric structure $(\mathbb{H}, ds_{-1}^{2})$, where $ds_{-1}^{2}=\frac{dx^{2}+dy^{2}}{y^{2}}$, $\Gamma$ is constituted by isometries of $(\mathbb{H}, ds_{-1}^{2})$.

Using normalized Fuchsian models, we can introduce a natural topology on $\mathcal{T}_{g}$. Given a Fuchsian model $(\Si, \Ga)$, by \cite[\S 2.5]{IT}, after conjugating in $\PSL(2, \R)$, there is a set of normalized generators $\{\al_{i}, \be_{i}\}^{g}_{i=1}$ for $\Ga$, where $\al_{g}$ has attractive fixed point at $1$ and $\be_{g}$ has repelling and attractive fixed point at $0$ and $\infty$ respectively. Moreover, this set of generators is uniquely determined by the equivalent class in $\mathcal{T}_{g}$. By \cite[\S 2.5]{IT}, $\al_{i}$, $\be_{i}$ can be uniquely written as $\al_{i}=\frac{a_{i}z+b_{i}}{c_{i}z+d_{i}}$, $a_{i}, b_{i}, c_{i}\in\R$, $c_{i}>0$, $a_{i}d_{i}-b_{i}c_{i}=1$, and $\be_{i}=\frac{a^{\pr}_{i}z+b^{\pr}_{i}}{c^{\pr}_{i}z+d^{\pr}_{i}}$, $a^{\pr}_{i}, b^{\pr}_{i}, c^{\pr}_{i}\in\R$, $c^{\pr}_{i}>0$, $a^{\pr}_{i}d^{\pr}_{i}-b^{\pr}_{i}c^{\pr}_{i}=1$, for $j=1,\cdots, g-1$. Hence we can define the \emph{Fricke coordinates}: $\mathcal{F}_{g}: \mathcal{T}_{g}\rightarrow\R^{6g-6}$ as $\mathcal{F}_{g}\big([\Si, f]\big)=(a_{i}, c_{i}, d_{i}, a_{i}^{\pr}, c_{i}^{\pr}, d_{i}^{\pr})_{i=1}^{g-1}$. By \cite[Theorem 2.25]{IT}, $\mathcal{F}_{g}$ is injective. Hence we have an induced topology on $\mathcal{T}_{g}$ by the Fricke coordinates.

\vspace{0.5em}
$\textbf{4}^{\circ}.$ We also need the notion of quasi-conformal maps. Let $f: \Sigma\rightarrow\Sigma^{\prime}$ be an orientation-preserving diffeomorphism between two Riemann surfaces. Given local complex coordinates $(z, \bar{z})$, $(w, \bar{w})$ on $\Sigma$ and $\Sigma^{\prime}$ respectively. Denote $f(z)=w\circ f\circ z$. The {\em Beltrami coefficient} is defined by
\begin{equation}\label{Beltrami coefficient}
\mu=\frac{f_{\bar{z}}}{f_{z}}.
\end{equation}
It is easy to see that $|\mu|$ does not depend on the local complex coordinates. If $|\mu|\leq k<1$, then we call such $f$ a \emph{quasi-conformal map}\footnote{When $|\mu|=0$, f is holomorphic.}.

Now let us combine the marked surface model with the quasi-conformal maps (see \cite[\S 5.1.2]{IT}). Let $\Sigma_{0}$ be a fixed Riemann surface, with a Fuchsian group $\Gamma_{0}$. After some conjugation in $\PSL(2, \mathbb{R})$, we can always assume that $(0, 1, \infty)$ are fixed by some elements in $\Gamma_{0}\setminus\{id\}$ \cite[\S 5.1.2]{IT}. We call such $\Gamma_{0}$ a \emph{normalized Fuchsian group}, and $(\Sigma_{0}, \Gamma_{0})$ a \emph{normalized Fuchsian model}. For any marked surface $(\Sigma, f)$, $f: \Sigma_{0}\rightarrow\Sigma$ is always a quasi-conformal map \cite[(1.4.2)]{IT}. Now we lift the quasi-conformal map $f$ up to the upper half space $\mathbb{H}$ by the covering maps $\pi_{0}: \mathbb{H}\rightarrow\Sigma_{0}$ and $\pi: \mathbb{H}\rightarrow\Sigma$ to get $\tilde{f}: \mathbb{H}\rightarrow\mathbb{H}$. After some $\PSL(2, \mathbb{R})$ action on the target $\mathbb{H}$, we can assume that $\tilde{f}$ also fixes the three points $(0, 1, \infty)$ (the uniqueness of such quasi-conformal maps is given in \cite[Proposition 4.33]{IT} and discussions in Proposition \ref{results about quasi-conformal maps}). We call such maps $\tilde{f}:\mathbb{H}\rightarrow\mathbb{H}$ \emph{canonical quasi-conformal maps}. By pushing over the normalized Fuchsian group $\Gamma_{0}$ on $\Sigma_{0}$ by $\tilde{f}$, we get another Fuchsian group $\Gamma_{\tilde{f}}=\tilde{f}\circ\Gamma_{0}\circ\tilde{f}^{-1}$, such that $\Sigma=\mathbb{H}/\Gamma_{\tilde{f}}$. Now for such a marking $f$, we can define an injective homeomorphism:
$$\theta_{\tilde{f}}: \Gamma_{0}\rightarrow PSL(2, \mathbb{R}),$$
where $\theta_{\tilde{f}}(\gamma)=\tilde{f}\circ\gamma\circ\tilde{f}^{-1}$, $\ga\in\Ga_0$. \cite[Lemma 5.1]{IT} showes that $(\Sigma_{1}, f_{1})$ and $(\Sigma_{2}, f_{2})$ are equivalent in the sense of Teichm\"uller space, if and only if $\theta_{\tilde{f}_{1}}=\theta_{\tilde{f}_{2}}$. Now we can define the following set:
\begin{equation}
\begin{split}
\mathcal{T}^{\sharp}_{g}= \big\{ & \theta_{\tilde{f}}:\ \tilde{f}\ \textrm{is a canonical quasiconformal map, such that}\\
                                 & \theta_{\tilde{f}}(\Gamma_{0})\ \textrm{is a Fuchsian group for some genus-$g$ surface}.\big\}\\
\end{split}
\end{equation}
\cite[Proposition 5.3]{IT} shows that $\mathcal{T}^{\sharp}_{g}$ is identified with the Teichm\"uller space $\mathcal{T}_{g}$. Later on, we will use this representation of the Teichm\"uller space $\mathcal{T}_g$, and we will extend the quasi-conformal maps to more general settings, say, in the Sobolev spaces.

\vspace{0.5em}
$\textbf{5}^{\circ}.$ We also need to introduce the Teichm\"uller mapping in a class of marked surfaces $[(\Sigma, f)]$, where $f:\Sigma_{0}\rightarrow\Sigma$ is an orientation-preserving diffeomorphism, hence is also a quasi-conformal map. By \cite[Theorem 5.9]{IT}, there exists a unique a holomorphic quadratic differential $\phi$ on $\Sigma_{0}$ with $||\phi||_{1}<1$\footnote{Here $||\phi||_{1}$ is the $L^{1}$-norm of $\phi$.}, and a unique quasi-conformal mapping $f_{1}:\Sigma_{0}\rightarrow\Sigma$ homotopic to $f$, such that the Beltrami coefficient $\mu_{f_{1}}$ (\ref{Beltrami coefficient}) of $f_{1}$ satisfies $\mu_{f_{1}}=\mu_{\phi}$, where
\begin{equation}
\mu_{\phi}\equiv||\phi||_{1}\frac{\bar{\phi}}{|\phi|}.
\end{equation}
We denote such a map by $f_{\phi}$ and call it \emph{Teichm\"uller mapping} \cite[\S 5.2.2]{IT}.

Denote the set of all holomorphic quadratic differentials on $\Sigma_{0}$ with $L^1$-norm $\|\cdot\|_{1}$ strictly less than one by $A_{2}(\Sigma_{0})_{1}$.
From \cite[Theorem 5.15]{IT}, we know that the mapping
$$\mathcal{F}: A_{2}(\Sigma_{0})_{1}\rightarrow\mathcal{T}_{g},$$
defined by $\mathcal{F}(\phi)=[(f_{\phi}(\Si_{0}), f_{\phi})]$, $\phi\in A_{2}(\Si_{0})_{1}$, is a homeomorphism, where $f_{\phi}$ is the unique Teichm\"uller mapping of the Beltrami coefficient $\mu_{\phi}$ in the class $[(f_{\phi}(\Si_{0}), f_{\phi})]$ of marked surfaces\footnote{The existence of $f_{\phi}$ can also be seen from the construction in \cite[\S 4.2]{IT}.}.
By the Riemann-Roch theorem, we know that $A_{2}(\Sigma_{0})_{1}$ is homotopic to a $(6g-6)$-dimensional Euclidean ball, hence is $\mathcal{T}_{g}$ and $\mathcal{T}^{\sharp}_{g}$. Later on, the topology on $\mathcal{T}_{g}$ and $\mathcal{T}^{\sharp}_{g}$ is identified with the topology on $A_{2}(\Sigma_{0})_{1}$.


\subsection{Some notations}\label{notations}

Now let us set down the framework of the variational method. Given a Riemannian manifold $(N, h)$. Let $\Sigma_{0}$ be a fixed Riemann surface of genus $g\geq 2$ with a normalized Fuchsian group $\Gamma_{0}$. Denote elements in the Teichm\"uller space $\mathcal{T}_{g}$ by $\tau$. Let $\phi_{\tau}\in A_{2}(\Sigma_{0})_{1}$ be the unique holomorphic quadratic differential on $\Sigma_{0}$ corresponding to $\tau$. Denote $f_{\tau}=f_{\phi_{\tau}}$ by the unique Teichm\"uller mapping determined by the Beltrami coefficient $\mu_{\phi_{\tau}}$ (\S \ref{Teichmuller space}.$\textbf{5}^{\circ}$), and $\tilde{f}_{\tau}:\mathbb{H}\rightarrow\mathbb{H}$ the unique canonical quasi-conformal mapping lifted up with respect to $\Gamma_{0}$. By \S \ref{Teichmuller space}.$\textbf{4}^{\circ}$, we can view $\tau$ as an equivalent class of marked surfaces $[(\Sigma_{\tau}, f_{\tau}\})]$ with normalized Fuchsian group $\Gamma_{\tau}=\theta_{\tilde{f}_{\tau}}(\Gamma_{0})$, i.e. $\Sigma_{\tau}=\mathbb{H}/\Gamma_{\tau}$.
\begin{definition}\label{variational space}
The variational spaces are defined as
\begin{equation}\label{Omega}
\Omega=\big\{\gamma(t)\in C^{0}\big([0, 1], C^{0}\cap W^{1,2}(\Sigma_{0}, N)\big)\big\},
\end{equation}
and
\begin{equation}\label{tilde{Omega}}
\tilde{\Omega}=\big\{(\gamma(t), \tau(t)):\ \gamma(t)\in C^{0}\big([0, 1], C^{0}\cap W^{1, 2}(\Sigma_{\tau(t)}, N)\big),\ \tau(t)\in C^{0}([0, 1], \mathcal{T}_{g})\big\},
\end{equation}
where $(\Sigma_{\tau}=\mathbb{H}/\Gamma_{\tau}, \Gamma_{\tau})$ is the normalized Fuchsian model corresponding to $\tau\in\mathcal{T}_{g}$. We always assume that the boundary $\gamma(0)$ and $\gamma(1)$ are mapped onto close curves in $N$.
\end{definition}

Now let us discuss the continuity of $\gamma(t)\in C^{0}\big([0, 1], C^{0}\cap W^{1, 2}(\Sigma_{\tau(t)}, N)\big)$. Here we can view all the $\gamma(t)$ as been defined on the upper half plane $\mathbb{H}$ lifted up by $\pi_{\tau(t)}:\mathbb{H}\rightarrow\Sigma_{\tau(t)}$, with the Fuchsian groups $\Gamma_{\tau(t)}$ varying continuously w.r.t.\footnote{Abbreviated for ``with respect to".} the parameter $t$. The continuity of $\gamma(t)$ w.r.t. $t$ can be defined as mappings on compact subsets $K$ of $\mathbb{H}$ with the Poincar\'e metric, i.e. $\gamma(t)\in C^{0}\big([0, 1], C^{0}\cap W^{1, 2}(K, N)\big)$. Another equivalent way to understand this is as follows. Let $\phi_{\tau(t)}$ be the holomorphic quadratic differentials corresponding to $\tau(t)$. The fact that $\tau(t)$ vary continuously w.r.t. $t$ is equivalent to that $\phi_{\tau(t)}$ vary continuously w.r.t. $t$ in $A_{2}(\Sigma_{0})_{1}$. Let $f_{\tau(t)}$ be the Teichm\"uller mappings corresponding to $\phi_{\tau(t)}$, then the canonical lift $\tilde{f}_{\tau(t)}:\mathbb{H}\rightarrow\mathbb{H}$ change continuously in $C_{loc}^{0}\cap W^{1, 2}(\mathbb{H}, \mathbb{H})$ by properties of quasi-conformal mapping\footnote{See \cite[Chap 4]{IT} and \S \ref{results about quasi-conformal maps}. Moreover, by \cite[Proposition 5.19]{IT}, $f_{\tau}$ is smooth away from zeros of $\phi_{\tau}$, and vary continuously in any $C^{k}$-norm w.r.t. $\tau$; also $f_{\tau}$ is uniformly Lipchitz when $\|\phi_{\tau}\|_1\leq k<1$.}. Using $f_{\tau(t)}$ as special markings for a continuous family of elements in $\mathcal{T}_{g}$, we can pull the path $\gamma(t): \Sigma_{\tau(t)}\rightarrow N$ back to $\Sigma_{0}$, i.e. $f_{\tau(t)}^{*}(\gamma(t))=\gamma(t)\circ f_{\tau(t)}:\Sigma_{0}\rightarrow N$. The continuity of $\gamma(t)$ w.r.t. $t$ is defined as the continuity of the path $f_{\tau(t)}^{*}\gamma(t)$ w.r.t. $t$ on the same surface $\Sigma_{0}$.

Next let us talk about the homotopy equivalence in $\tilde{\Omega}$. Consider two elements $\big\{(\gamma_{i}(t), \tau_{i}(t)):\ i=1, 2\big\}$. They have different domains $\Sigma_{\tau_{i}(t)},\ i=1, 2$ given by normalized Fuchsian models $\Gamma_{\tau_{i}(t)}$. As above, we use Teichm\"uller mappings $f_{\phi_{\tau_{i}(t)}}:\Sigma_{0}\rightarrow\Sigma_{\tau_{i}(t)},\ i=1, 2$ to identify $\Sigma_{\tau_{i}(t)},\ i=1, 2$ with $\Sigma_{0}$, where $\phi_{\tau_{i}(t)}$ are the holomorphic quadratic differentials corresponding to $\tau_{i}(t)$, $i=1, 2$. Since $\mathcal{T}_{g}$ is homotopic to a ball, $\tau_{1}(t)$ and $\tau_{2}(t)$ are always homotopic to each other. Hence we say that $\big\{(\gamma_{1}(t), \tau_{1}(t))\big\}$ is homotopic to $\big\{(\gamma_{2}(t), \tau_{2}(t)\big\}$ if $f_{\phi_{\tau_{1}(t)}}^{*}\gamma_{1}(t)$ is homotopic to $f_{\phi_{\tau_{2}(t)}}^{*}\gamma_{2}(t)$.

\begin{definition}\label{W and W-E}
Fix a homotopy class $[\beta]\subset\Omega$, and $\tau_{0}$ a fixed element in $\mathcal{T}_{g}$ given by $[(\Sigma_{0}, id)]$. For area functional, define
\begin{equation}\label{W}
\mathcal{W}=\underset{\rho\in[\beta]}{\inf}\underset{t\in[0,1]}{\max}Area\big(\rho(t)\big).
\end{equation}
For energy functional, define
\begin{equation}\label{W-E}
\mathcal{W}_{E}=\underset{(\rho, \tau)\in[(\beta, \tau_{0})]}{\inf}\underset{t\in[0,1]}{\max}E\big(\rho(t), \tau(t)\big).
\end{equation}
\end{definition}

\begin{remark}
Later, we will show that $\mathcal{W}=\mathcal{W}_{E}$ in Remark \ref{energy equivalent to area}. We will mainly focus on the case when $\mathcal{W}>0$.
\end{remark}


\subsection{Sketch of the variational method}\label{sketch of variational method}

Now a natural question is to find the critical points corresponding to $\mathcal{W}$. In fact, the critical points are achieved by some conformal harmonic mappings from surfaces degenerated from $\Sigma_{0}$ together with finitely many harmonic spheres. To achieve the critical points, we use the geometric variational method. We take a minimizing sequence $\big\{\tilde{\gamma}_{n}(t)\big\}_{n\in\N}\subset[\beta]\subset\Omega$, such that
$$\lim_{n\rightarrow\infty}\underset{t\in[0,1]}{max}Area\big(\tilde{\gamma}_{n}(t)\big)=\mathcal{W}.$$
In fact, by the standard mollification method \cite[\S D.1]{CM3}\cite[\S 4]{ScU}, we can assume that $\tilde{\gamma}_{n}(t)$ vary continuously in $C^{2}$-class, i,e. $\tilde{\gamma}_{n}(t)\in C^{0}\big([0, 1], C^{2}(\Sigma_{0}, N)\big)$.

Then we would like to change to use the variational method of the energy functional $E$ and hence work in $\tilde{\Omega}$. The variational method consists of the following three steps. \textbf{First}, we do almost conformal reparametrizations to module out the conformal group action. Pull back the ambient metric $\tilde{g}_{n}(t)=\tilde{\gamma}_{n}(t)^{*}h$. We want to show that $\tilde{g}_{n}(t)$, which may be degenerate, determine a family of elements $\tau_{n}(t)\in\mathcal{T}_{g}$. Suppose that the corresponding normalized Fuchsian model and Teichm\"uller mappings of $\tau_n(t)$ are $(\Sigma_{\tau_{n}(t)}, \Gamma_{\tau_{n}(t)}, f_{\tau_{n}(t)})$, where $\Gamma_{\tau_{n}(t)}=\theta_{\tilde{f}_{\tau_{n}(t)}}(\Gamma_{0})$ and $\Sigma_{\tau_{n}(t)}=\mathbb{H}/\Gamma_{\tau_{n}(t)}$. We want to find almost conformal parametrizations $h_{n}(t): \Sigma_{\tau_{n}(t)}\rightarrow (\Sigma_{0}, \tilde{g}_{n}(t))$, such that the reparametrization $\big(\gamma_{n}(t), \tau_{n}(t)\big)=\Big(\tilde{\gamma}_{n}\big(h_{n}(t), t\big), \tau_{n}(t)\Big)\in\big[\big(\tilde{\gamma}_{n}(t), \tau_{0}\big)\big]$ have energy close to area, i.e. $E\big(\gamma_{n}(t), \tau_{n}(t)\big)-Area\big(\gamma_{n}(t)\big)\rightarrow 0$ as $n\rightarrow\infty$. \textbf{Second}, we do compactification by deforming $\gamma_{n}(t)$ to $\rho_{n}(t)$. We will adapt the local harmonic replacement method developed by Colding and Minicozzi \cite{CM3, Z} to the hyperbolic surfaces. We make $\rho_{n}(t)$ to be almost harmonic mappings, so as to get bubble tree compactness as in \cite{SU1, CM3, Z}. \textbf{Finally}, we discuss the degenerations of conformal structures of $\tau_{n}(t)$. We will show that $\big(\rho_{n}(t), \tau_{n}(t)\big)$ bubble tree converge to certain conformal harmonic mappings defined on surfaces degenerated from $\Sigma_{0}$ together with some harmonic spheres, and we will prove the energy identity, hence show that the sum of the area\footnote{The area equals to the energy since the final targets are all conformal.} equals to $\mathcal{W}$ (\ref{W}).

\vspace{0.5em}
In the following sections, we will discuss the three steps in details.


\section{Conformal parametrization in the high genus case}\label{Conformal parametrization}

In this section, we will do almost conformal re-parametrization for the minimizing sequence $\{\tilde{\gamma}_{n}(t)\}_{n\in\N}\subset\Omega$. We can assume that $\{\tilde{\gamma}_{n}(t)\}$ have better regularity.

\begin{lemma}\label{mollifying}
(\cite[Lemma D.1]{CM3}, \cite[Lemma 3.1]{Z}) Suppose $\tilde{\gamma}_{n}(t)$ are chosen as in the above section, we can perturb them to
get a new minimizing sequence in the same homotopy class $[\beta]$, such that (denoting them still as $\tilde{\gamma}_{n}(t)$), $\tilde{\gamma}_{n}(t)\in C^{0} \Big([0,1],C^{2}(\Sigma_{0}, N)\Big)$.
\end{lemma}


\subsection{Summary of results on quasi-conformal mappings}

Before going to the uniformization and re-parametrization, we first summarize results of quasi-conformal mappings proved in \cite{AB, IT} and the appendix of \cite{Z}. We will focus on the a priori estimates for the conformal diffeomorphism between general metrics.

\subsubsection{Results about quasi-conformal maps}\label{results about quasi-conformal maps}

We mainly refer to Ahlfors and Bers in \cite{AB} (see also \cite[Section 6.1]{Z}). They gave the \textbf{existence} and \textbf{uniqueness} of conformal diffeomorphism $f^{\mu}:\mathbb{C}_{|dz+\mu d\overline{z}|^{2}}\rightarrow\mathbb{C}_{dwd\overline{w}}$\footnote{We use $\{z, \overline{z}\}$ and $\{w, \overline{w}\}$ as complex coordinates on $\mathbb{C}$.} fixing three points $(0, 1, \infty)$ for any $L^{\infty}$-function $\mu$ with $|\mu|\leq k<1$ (see also \cite[Theorem 4.30, Proposition 4.33]{IT}). We also call such $\mu$ (generalized) Beltrami coefficient here\footnote{Compared to that in \S \ref{Teichmuller space}.$\textbf{4}^{\circ}$, this $\nu$ is not invariant under Fuchsian group.}. Such maps satisfy the following equation (see \cite[(57)]{Z}):
\begin{equation}\label{cr1}
f^{\mu}_{\overline{z}}=\mu(z)f^{\mu}_{z}.
\end{equation}

Define the function space $B_{p}(\mathbb{C})=C^{1-\frac{2}{p}}\cap W^{1,p}_{loc}(\mathbb{C})$, where $p>2$ depends only on the bound $k$ of $|\mu|$. Suppose $\mu, \nu \in L^{\infty}(\mathbb{C})$, and $|\mu|,|\nu|\leq k$, with $k<1$. Let $f^{\mu}, f^{\nu}$ be the corresponding conformal homeomorphisms, then:

\begin{lemma}\label{cont1}
(\cite[Lemma 16, Theorem 7, Lemma 17, Theorem 8]{AB}, \cite[Lemma 6.2]{Z})
\begin{equation}
d_{S^{2}}\big(f^{\mu}(z_{1}), f^{\mu}(z_{2})\big)\leq c d_{S^{2}}(z_{1}, z_{2})^{\alpha},
\end{equation}
\begin{equation}
\|f^{\mu}_{z}\|_{L^{p}(B_{R})}\leq c(R),
\end{equation}
\begin{equation}
d_{S^{2}}\big(f^{\mu}(z), f^{\nu}(z)\big)\leq C \|\mu-\nu\|_{\infty},
\end{equation}
\begin{equation}
\|(f^{\mu}-f^{\nu})_{z}\|_{L^{p}(B_{R})}\leq C(R)\|\mu-\nu\|_{\infty}.
\end{equation}
Here $d_{S^{2}}$ is the sphere distance, which is equivalent to the plane distance of $\mathbb{C}$ on compact sets. $\alpha=1-\frac{2}{p}$. $B_{R}$ is a disk of radius $R$ on $\mathbb{C}$. All constants are uniformly bounded depending only on $k<1$.
\end{lemma}

\subsubsection{Results about quasi-linear quasi-conformal maps}\label{results about quasi-linear quasi-conformal maps}

What we concern in our case are the conformal homeomorphisms $h^{\mu}:\mathbb{C}_{dwd\overline{w}}\rightarrow\mathbb{C}_{|dz+\mu
d\overline{z}|^{2}}$ fixing three points $(0,1, \infty)$, which arise as the inverse mappings of those $f^{\mu}$ of Ahlfors and Bers. In fact, suppose
\begin{equation}\label{quasi-linear quasi-conformal mapping}
h^{\mu}(w)=(f^{\mu})^{-1}(w),
\end{equation}
then our mappings satisfy:
\begin{equation}\label{cr2}
h^{\mu}_{\overline{w}}=-\mu(h^{\mu}(w))\overline{h^{\mu}_{w}}.
\end{equation}
Since the equation is quasi-linear (compared to linear equation (\ref{cr1})), we call such $h^{\mu}$ \emph{quasi-linear quasi-conformal maps}.

If $\{\mu_n\}$ are a sequence of Beltrami coefficients as above, such that $\|\mu_{n}-\mu\|_{C^{1}}\rightarrow 0$, and $h^{\mu_{n}}$ satisfying (\ref{quasi-linear quasi-conformal mapping}), we have results similar to the above:

\begin{lemma}\label{cont2}
(\cite[Lemma 6.3]{Z})
\begin{equation}\label{sphere convergence for the inverse of u-conformal map}
d_{S^{2}}\big(h^{\mu_{n}}, h^{\mu}\big)\rightarrow 0,
\end{equation}
\begin{equation}\label{local Lp convergence of u-conformal map}
\|(h^{\mu_{n}}-h^{\mu})_{w}\|_{L^{p}(B_{R})}\rightarrow 0,
\end{equation}
where $p$ is given in Lemma \ref{cont1}.
\end{lemma}


\subsection{Uniformization for surfaces of genus $g\geq 2$}

Fix $\Sigma_{0}$ with normalized Fuchsian model $\Gamma_{0}$ as before. Denote $\pi_{0}: \mathbb{H}\rightarrow \Si_0$ by the quotient map for $(\Sigma_{0}, \Gamma_{0})$. Denote the Poincar\'e metric on $\Sigma_{0}$ by $g_{0}$. Given $\tau\in\mathcal{T}_{g}$, let the corresponding normalized Fuchsian model be $(\mathbb{H}, \Gamma_{\tau}, \Sigma_{\tau})$ as in the beginning of \S \ref{notations}. Let $\pi_{\tau}: \mathbb{H}\rightarrow\Sigma_{\tau}$ be the quotient map, and $f_{\tau}: \Sigma_{0}\rightarrow\Sigma_{\tau}$ the Teichm\"uller mapping.

\begin{proposition}\label{uniformization proposition}
Let $g$ be a $C^{1}$ metric on $\Sigma_{0}$. We can view $g$ as a metric on $\mathbb{H}$ by lifting up using $\pi_{0}$. Then there is a unique element $\tau\in\mathcal{T}_{g}$ with normalized Fuchsian model $(\Sigma_{\tau}, \Gamma_{\tau})$, and a unique orientation-preserving $C^{1,\frac{1}{2}}$ conformal diffeomorphism $h:\Sigma_{\tau}\rightarrow(\Sigma_{0}, g)$, such that $h$ is homotopic to $f_{\tau}^{-1}$, with the normalization that if lifting up to $\tilde{h}:\mathbb{H}\rightarrow\mathbb{H}$ by $\pi_{\tau}$ and $\pi_{0}$, $\tilde{h}^{*}(\Gamma_{0})=\Gamma_{\tau}$. Furthermore, given a one-parameter family of $C^{1}$ metrics $g(t)$ on $\Sigma_{0}$ which is continuous w.r.t. $t$ in the $C^{1}$-class, i.e. $g(t)\in C^{1}\big([0,1],\ \textrm{$C^{1}$-metrics}\big)$, and $g(t)\geq\epsilon g_{0}$ for some uniform $\epsilon>0$, let $\big(\tau(t), h(t)\big)$ be the corresponding elements in $\mathcal{T}_{g}$ and normalized conformal diffeomorphisms, then $\tau(t)$ and $h(t)$ are continuously w.r.t. $t$ in $\mathcal{T}_{g}$ and $C^{0}\cap W^{1,2}(\Sigma_{\tau(t)}, \Sigma_{0})$ respectively.
\end{proposition}

\begin{remark}
Here the space $C^{0}\cap W^{1,2}(\Sigma_{\tau(t)}, \Sigma_{0})$ have varying domains $\Sigma_{\tau(t)}$, and the continuity is defined in \S \ref{notations}.
\end{remark}

We need the following result to prove the proposition. Let $g$ be a Riemannian metric on the complex plane $\mathbb{C}$.

\begin{lemma}\label{domain metric}
(\cite[Lemma 6.1]{Z}) In the complex coordinates $\{z, \overline{z}\}$, we can write $g=\lambda(z)|dz+\mu(z)d\overline{z}|^{2}$. Here $\lambda(z)>0$, and $\mu(z)$ is complex function on the complex plane with $|\mu|<1$. If $g\geq\epsilon dzd\overline{z}$, there exists a $k=k(\epsilon)<1$, such that $|\mu|\leq k$. Furthermore, $\mu$ is a rational function of the components $g_{ij}(z)$, so if a family $g(t)$ is continuous w.r.t. $t$ in the $C^{1}$-class, the corresponding $\mu(t)$ is also continuous in the $C^{1}$-class.
\end{lemma}

\begin{proof}
(of Proposition \ref{uniformization proposition}). Let us fist show the existence of such mark $\tau\in\mathcal{T}_{g}$ and conformal homeomorphism $h$. Pull $g$ back to $\mathbb{H}$ by $\pi_{0}$ and denote it still by $g$, then it is invariant under the $\Gamma_{0}$ group action. By Lemma \ref{domain metric}, $g=\lambda(z)|dz+\mu(z)d\overline{z}|^{2}$, with $|\mu(z)|\leq k<1$. Here $\mu$ is the Beltrami coefficient mentioned in \S \ref{results about quasi-conformal maps}. Then we have a unique normalized quasi-conformal mapping $f^{\mu}: \mathbb{H}_{|dz+\mu d\overline{z}|^{2}}\rightarrow\mathbb{H}_{dwd\overline{w}}$ (see also \cite[Proposition 4.33]{IT}). Now push forward the Fuchsian group $\Gamma_{0}$ under $f^{\mu}$. Since $f^{\mu}$ is a homeomorphism, we get another Fuchsian group $\Gamma_{f^{\mu}}=f^{\mu}_{*}(\Gamma_{0})=\theta_{f^{\mu}}(\Gamma_{0})$ on $\mathbb{H}_{dwd\overline{w}}$. This Fuchsian group gives a normalized Fuchsian model which represents an element in $\mathcal{T}_{g}$. Denote this element by $\tau$. Denoting $\Gamma_{f^{\mu}}$ by $\Gamma_{\tau}$, we get a Fuchsian model $\Sigma_{\tau}=\mathbb{H}/\Gamma_{\tau}$. Let $\pi_{\tau}: \mathbb{H}\rightarrow\Sigma_{\tau}$ be the quotient map, then after taking quotient of $f^{\mu}$ by $\pi_{0}$ and $\pi_{\tau}$, we get $f^{\mu}:\Sigma_{0}\rightarrow\Sigma_{\tau}$\footnote{We denote the quotient map still by $f^{\mu}$.}. By the definition of quasi-conformal maps, this $f^{\mu}$ is conformal between $(\Sigma_{0}, |dz+\mu(z)d\overline{z}|^{2})$ and $\Sigma_{\tau}$, and hence conformal between $(\Sigma_{0}, g)$ and $\Sigma_{\tau}$. Let $h=(f^{\mu})^{-1}$, then $h$ is a conformal homeomorphism between $\Sigma_{\tau}$ and $(\Sigma_{0}, g)$. The $C^{1, \frac{1}{2}}$-regularity of $h$ follows from \cite[Theorem 3.1.1 and Theorem 3.3.1]{J1}. By the definition of Teichm\"uller map $f_{\tau}: \Sigma_{0}\rightarrow\Sigma_{\tau}$, if we pull $f_{\tau}$ back to $\tilde{f}_{\tau}:\mathbb{H}\rightarrow\mathbb{H}$ by $\pi_{0}$ and $\pi_{\tau}$, then $(\tilde{f}_{\tau})_{*}(\Gamma_{0})=\theta_{f_{\tau}}(\Gamma_{0})=\Gamma_{\tau}$. So by \cite[Lemma 5.1]{IT}, we know that $f_{\tau}$ is homotopic to $f^{\mu}$. So $h$ is homotopic to $f_{\tau}^{-1}$. The normalization of $\tilde{h}$, i.e. $\tilde{h}^{*}(\Gamma_{0})=\Gamma_{\tau}$, comes trivially from the fact that $\Gamma_{\tau}=(f^{\mu})_{*}(\Gamma_{0})$ and $\tilde{h}=(f^{\mu})^{-1}$. The uniqueness of such $\tau$ and $h$ follows from the uniqueness of $f^{\mu}$.

Now let us talk about the continuous dependence of $(\tau, h)$ on $\mu$. For a continuous family of $C^{1}$ metrics $g(t)$, after pulling back to $\mathbb{H}$ by $\pi_{0}$, $g(t)=\lambda(t)|dz+\mu(t)d\overline{z}|^{2}$, and is continuous w.r.t. $t$ in the $C^{1}$-class. We have $|\mu(t)|\leq k(\epsilon)<1$, and $\mu(t)$ continuous w.r.t. $t$ in the $C^{1}$ class by Lemma \ref{domain metric}. Let $f(t)=f^{\mu(t)}$ and $\tilde{h}(t)=(f(t))^{-1}$ as above.

First, let us show the continuity of $\tau(t)$ w.r.t. the parameter $t$. Now the corresponding normalized Fuchsian model $\Gamma_{\tau(t)}$ is given by $f^{\mu(t)}_{*}(\Gamma_{0})$. Suppose that the normalized generators for $\Gamma_{0}$ (see \S \ref{Teichmuller space}.$\textbf{3}^{\circ}$ and \cite[\S 2.5]{IT}) are $\{\alpha^{0}_{i}, \beta^{0}_{i}\}^{g}_{i=1}$, where $\alpha^{0}_{g}$ has attractive fixed point at $1$ and $\beta^{0}_{g}$ has repelling and attractive fixed point at $0$ and $\infty$ respectively. Then clearly $\{\theta_{f^{\mu(t)}}(\alpha^{0}_{i}), \theta_{f^{\mu(t)}}(\beta^{0}_{i})\}^{g}_{i=1}$ form the normalized generators for $\Gamma_{\tau(t)}$. Now
\begin{equation}
\theta_{f^{\mu(t)}}(\gamma)=f^{\mu(t)}\circ\gamma\circ(f^{\mu(t)})^{-1}=f^{\mu(t)}\circ\gamma\circ\tilde{h}(t).
\end{equation}
By Lemma \ref{cont1} and Lemma \ref{cont2}, $f^{\mu(t)}$ and $\tilde{h}(t)$ are continuous w.r.t. the parameter $t$ in $C^{0}$-class when acting on compact subsets of $\mathbb{C}$. So for fixed $\gamma\in\Gamma_{0}$, $\theta_{f^{\mu(t)}}(\gamma)$ is continuous w.r.t. the parameter $t$, which means that the coefficients of the linear fractional transformation corresponding to $\theta_{f^{\mu(t)}}(\gamma)$ are continuous functions of $t$. So the coefficients for $\{\theta_{f^{\mu(t)}}(\alpha^{0}_{i}), \theta_{f^{\mu(t)}}(\beta^{0}_{i})\}^{g}_{i=1}$ are continuous functions of $t$. Now using the topology of Fricke Space as in \S \ref{Teichmuller space}.$\textbf{3}^{\circ}$ (see also \cite[Section 2.5, Lemma 5.10 and Lemma 5.13]{IT}), the corresponding elements $\tau(t)\in\mathcal{T}_{g}$ are continuous w.r.t. the parameter $t$ in the natural topology of $\mathcal{T}_{g}$.

Next, let us show the continuity of $h(t)$. Lift up to $\tilde{h}(t):\mathbb{H}_{dwd\overline{w}}\rightarrow\mathbb{H}_{|dz+\mu(t)d\overline{z}|^{2}}$, then $\tilde{h}(t)=(f^{\mu(t)})^{-1}$ are $\mu(t)$-quasi-linear quasi-conformal map as in \S \ref{results about quasi-linear quasi-conformal maps}. So by Lemma \ref{cont2}, we have the local $C^{0}\cap W^{1, 2}(\mathbb{H}, \mathbb{H})$ continuity of $\tilde{h}(t)$ w.r.t. $t$, since $\mu(t)$ is continuous in $C^{1}$ w.r.t. the parameter $t$. It directly implies the continuity of $h(t):\Sigma_{\tau(t)}\rightarrow\Sigma_{0}$ in the sense of \S \ref{notations}, i.e. when restricting to compact subsets $K$ of $\mathbb{H}$, the lift-up mapping $\tilde{h}(t)\in C^{0}\big([0, 1], C^{0}\cap W^{1, 2}(K, N)\big)$.
\end{proof}


\subsection{Construction of the conformal re-parametrization}

Recall the minimizing sequence $\big\{\tilde{\gamma}_{n}(t)\big\}_{n\in\N}\subset[\beta]\subset\Omega$ given in \S \ref{sketch of variational method}. We consider $\tilde{g}_{n}(t)=\tilde{\gamma}_{n}(t)^{*}h$, which is continuous w.r.t. $``t"$ in the $C^{1}$-class by Lemma \ref{mollifying}. Since $\tilde{g}_{n}(t)$ may be degenerate, let $g_{n}(t)=\tilde{g}_{n}(t)+\delta_{n} g_{0}$, where $g_{0}$ is the Poincar\'e metric of $\Sigma_{0}$, and $\delta_{n}$ is arbitrarily small. Then $g_{n}(t)$ uniquely determines $\tau_{n}(t)\subset\mathcal{T}_{g}$ and conformal diffeomorphism $h_{n}(t)$ by Proposition \ref{uniformization proposition}. We have the following result similar to \cite[Theorem 3.1]{Z}.

\begin{theorem}\label{conformal parametrization}
Using the above notations, we have re-parametrizations $\big(\gamma_{n}(t), \tau_{n}(t)\big)\in\tilde{\Omega}$ for $\tilde{\gamma}_{n}(t)$, i.e.
$\gamma_{n}(t)=\tilde{\gamma}_{n}\big(h_{n}(t), t\big)$, such that $\gamma_{n}(t)\in\big[\tilde{\gamma}_{n}\big]$ in $\ti{\Omega}$, and
\begin{equation}\label{equation of conformal parametrization}
E\big(\gamma_{n}(t),
\tau_{n}(t)\big)-Area\big(\gamma_{n}(t)\big)\rightarrow 0,
\end{equation}
for some sequence $\delta_{n}\rightarrow 0$ as $n\rightarrow\infty$.
\end{theorem}
\begin{proof}
We know that $h_{n}(t): \Sigma_{\tau_{n}(t)}\rightarrow(\Sigma, g_{n}(t))$ are conformal diffeomorphisms. Let $\gamma_{n}(t)=\tilde{\gamma}_{n}\big(h_{n}(t), t\big): \Sigma_{\tau_{n}(t)}\rightarrow N$ be the composition with the almost conformal parametrization. To show that $\gamma_{n}(t)$ is a sweep-out in $\ti{\Omega}$, we only need to show the continuity. The continuity of $t\rightarrow\gamma_{n}(t)$ from $[0,1]$ to $C^{0}\cap W^{1,2}(\Sigma_{\tau_{n}(t)}, N)$ follows from the continuity of $t\rightarrow\tilde{\gamma}_{n}(t)$ in $C^{2}$ by Lemma \ref{mollifying}, and that of $t\rightarrow h_{n}(t)$ in $C^{0}\cap W^{1,2}(\Sigma_{\tau(t)}, \Sigma_{0})$ by Proposition \ref{uniformization proposition}.

Moreover, $\gamma_{n}(t)$ is homotopic to $\tilde{\gamma}_{n}(t)$ by the following argument. From our discussion of homotopy equivalence of mappings defined on different domains in \S \ref{notations}, we view $\gamma_{n}(t)$ as mappings defined on $\Sigma_{0}$ by composing with the Teichm\"uller mapping $f_{\tau_{n}(t)}: \Sigma_{0}\rightarrow\Sigma_{\tau_{n}(t)}$, and then compare it to $\tilde{\gamma}_{n}(t)$. Since $h_{n}(t)$ are homotopic equivalent to $f_{\tau_{n}(t)}^{-1}$ by Proposition \ref{uniformization proposition}, $h_{n}(t)\circ f_{\tau_{n}(t)}$ is homotopic equivalent to the identity map of $\Sigma_{0}$. While $\gamma_{n}$ are the composition of $\tilde{\gamma}_{n}$ with $h_{n}(t)$, then $\gamma_{n}\circ f_{\tau_{n}}$ is homotopic equivalent to $\tilde{\gamma}_{n}$, hence $\gamma_{n}\sim\tilde{\gamma}_{n}$.

Finally, we can get estimates as in \cite[Appendix D]{CM3} and the proof of \cite[Theorem 3.1]{Z}:
\begin{equation}
\begin{split}
E\big(\gamma_{n}(t), \tau_{n}(t)\big) &=E\big(h_{n}(t):T^{2}_{\tau_{n}(t)}\rightarrow(\Sigma_{0}, \tilde{g}_{n}(t))\big)\leq E\big(h_{n}(t): \Sigma_{\tau_{n}(t)}\rightarrow(\Sigma_{0}, g_{n}(t))\big)\\
&=Area\big(h_{n}(t): \Sigma_{\tau_{n}(t)}\rightarrow(\Sigma_{0}, g_{n}(t))\big)\\
&=Area\big(\Sigma_{0}, g_{n}(t)\big)=\int_{\Sigma_{0}}[det\big(g_{n}(t)\big)]^{\frac{1}{2}}dvol_{0}\\
&=\int_{\Sigma_{0}}[det\big(\tilde{g}_{n}(t)\big)+\delta_{n}Tr_{g_{0}}\tilde{g}_{n}(t)+C(\tilde{g}_{n}(t))\delta_{n}^{2}]^{\frac{1}{2}}dvol_{0}\\
&\leq Area(\Sigma_{0}, \tilde{g}_{n}(t))+C(\tilde{g}_{n}(t))\sqrt{\delta_{n}}\\
&=Area\big(\gamma_{n}(t):\Sigma_{0}\rightarrow N\big)+C(\tilde{\gamma}_{n})\sqrt{\delta_{n}}.
\end{split}
\end{equation}
The first and last equality follow from the definition of energy and area integral, and the first inequality is due to the fact
$\tilde{g}_{n}(t)\leq g_{n}(t)$. Hence we have (\ref{equation of conformal parametrization}), if we choose $\delta_{n}\rightarrow 0$ depending only on $\tilde{\gamma}_{n}$.
\end{proof}

\begin{remark}\label{energy equivalent to area}
By argument similar to \cite[Proposition 1.5]{CM3} and \cite[Remark 3.2]{Z}, the above theorem implies that $\mathcal{W}=\mathcal{W}_{E}$.
\end{remark}


\section{Compactification for mappings}\label{compactification for mappings}

For each $(\gamma_{n}(t), \tau_{n}(t))$ gotten above, $\tau_{n}(t)$ corresponds to a normalized Fuchsian model $(\Sigma_{\tau_{n}(t)}, \Gamma_{\tau_{n}(t)})$. We can also view $\gamma_{n}(t)$ as been lifted up to $\mathbb{H}$ by $\pi_{\tau_{n}(t)}: \mathbb{H}\rightarrow\Sigma_{\tau_{n}(t)}$. Denote the lifted mappings again by $\gamma_{n}(t)$, then $\gamma_{n}(t)$ can be viewed as defined on the same domain $\mathbb{H}$, i.e. $\gamma_{n}(t):\mathbb{H}\rightarrow N$, but invariant under different Fuchsian groups $\Gamma_{\tau_{n}(t)}$ action, i.e. $\forall\gamma\in\Gamma_{\tau_{n}(t)}$, $\gamma_{n}(t)\circ\gamma=\gamma_{n}(t)$. We can apply similar perturbation procedure to the lifted mappings as in \cite{CM3}\cite{Z}.

Before doing such perturbations, we need to introduce the notion of {\em collections of disjoint balls} on $\Sigma_{\tau}$. Here we use $\mathcal{B}=\cup_{i=1}^{n}B_{i}$ to denote a finite collection of disjoint geodesic balls on $\Sigma_{\tau}$, with the radii of each ball less than the injective radius of the center of that ball on $\Sigma_{\tau}$. Taking a ball $B\in\mathcal{B}$ with radius $r_{B}$, we will use a sub-geodesic ball with the same center but with the radius only a ratio $\mu<1$ of $r_{B}$, which we denote by $\mu B$. Such a geodesic ball $B$ with hyperbolic metric of curvature $-1$ can always be pulled back to the Poincar\'e disk $(D, ds^{2}_{-1}=\frac{|dx|^{2}}{(1-|x|^{2})^{2}})$, such that the center of $B$ goes to the center of $D$. Then $B$ can be viewed as a disk $B(0, r^{0}_{B})$\footnote{We will use $B(0, r^0)$ to denote a disk center at $0$ of Enclidean radius $r^0$ in the following.} in $D$ with hyperbolic metric $ds^{2}_{-1}$, where $r^{0}_{B}$ is the Euclidean radius of the image of $B$ and $r_{B}=\int_{0}^{r^{0}_{B}}\frac{1}{1-t^{2}}dt=\tanh^{-1}(r^{0}_{B})$. The hyperbolic metric is now conformal and uniformly equivalent to the Euclidean metric $ds^{2}_{0}=|dx|^{2}$ on $B$. Here \emph{uniformly equivalent} means $ds^{2}_{0}\leq ds^{2}_{-1}\leq Cds^{2}_{0}$ for some constant $C>1$. There exists a small number:
\begin{equation}\label{small radius}
r_{0}=\tanh^{-1}(\frac{1}{2}),
\end{equation}
such that if we restrict the radius $r_{B}$ of $B$ with $r_B\leq r_{0}$, we can choose the constant $C=\frac{16}{9}$. Then if we consider $\frac{1}{4}B$, under the Euclidean metric $ds^{2}_{0}$, the radius of $\frac{1}{4}B$ is less than $\frac{1}{2}r^{0}_{B}$, i.e. $\frac{1}{4}B\subset B(0, \frac{1}{2}r^{0}_{B})$. Later on, we will always assume that the  geodesic balls have their radii bounded from above by $r_{0}$.

\vspace{0.5em}
Now we state the main deformation lemma.

\begin{lemma}\label{compactification}
Let $[\beta]$ and $\mathcal{W}_{E}$ be as in Definition \ref{W and W-E}. For any $\big(\gamma(t), \tau(t)\big)\in[\beta]\subset\tilde{\Omega}$ with $\underset{t\in[0,1]}{max}E\big(\gamma(t), \tau(t)\big)-\mathcal{W}_{E}\ll 1$, if $\big(\gamma(t), \tau(t)\big)$ is not harmonic unless $\gamma(t)$ is a constant map, we can perturb $\gamma(t)$ to $\rho(t)$, such that $\rho(t)\in[\gamma(t)]$ and $E\big(\rho(t), \tau(t)\big)\leq E\big(\gamma(t), \tau(t)\big)$. Moreover for any $t$ such that $E\big(\gamma(t),\tau(t)\big)\geq\frac{1}{2}\mathcal{W}_{E}$, $\rho(t)$ satisfy:
\begin{itemize}
\addtolength{\itemsep}{-0.7em}

\item[(*)] For any finite collection of disjoint balls $\underset{i}{\cup}B_{i}$ on $\Sigma_{\tau(t)}$ with the geodesic radius of each ball $B_{i}$ bounded above by $r_0$ and the injective radius of the center of $B_i$ on $\Si_{\tau(t)}$, such that $E\big(\rho(t), \underset{i}{\cup}B_{i}\big)\leq\epsilon_{0}$, let $v$ be the energy minimizing harmonic map with the same boundary value as $\rho(t)$ on $\frac{1}{64}\underset{i}{\cup}B_{i}$, then we have:
\begin{equation}\label{compactification formula}
\int_{\frac{1}{64}\underset{i}{\cup}B_{i}}|\nabla\rho(t)-\nabla v|^{2}\leq\Psi\Big(E\big(\gamma(t), \tau(t)\big)-E\big(\rho(t), \tau(t)\big)\Big).
\end{equation}
\end{itemize}
Here $\epsilon_{0}$ is some small constant, and $\Psi$ is a positive continuous function with $\Psi(0)=0$.
\end{lemma}

\begin{remark}
We will mainly use the idea in the proof of \cite[Theorem 2.1]{CM3} and \cite[Lemma 4.1]{Z}. As discussed in the remarks following \cite[Lemma 4.1]{Z}, we would need to show the continuity of local harmonic replacement and comparison of energy decrease of successive harmonic replacements. The continuity of harmonic replacement is a conformal invariant property, which can be handled by pulling every ball we care back to the center of the Poincar\'e disk as above. For the comparison of the energy decrease, it turns out that what we really need to care is the analysis on a single ball. So we could do that by pulling the chosen ball to the center of the Poincar\'e disk again, without caring about the image of the other balls.
\end{remark}

In the following three subsections, we first list the results about analysis of harmonic replacements on disks. Then we give a result of comparison of harmonic replacements, where we show a result similar to \cite[Lemma 3.11]{CM3} and \cite[Lemma 4.2]{Z} by adapting the proof to the hyperbolic surfaces. At the end, we give the deformation map $\gamma\rightarrow\rho$ by explicit constructions.


\subsection{Results about harmonic replacements on disks}

Here we summarize some known results of harmonic replacements on disks. Let $B_{1}$ be the unit disk in $\mathbb{R}^{2}$, and $N$ the ambient manifold.

\begin{theorem}\label{energy gap}
(\cite[Theorem 3.1]{CM3}) There exists a small constant $\epsilon_{1}$ (depending only on $N$) such that for all maps $u, v\in W^{1,2}(B_{1}, N)$ , if $v$ is weakly harmonic with the same boundary value as $u$, and $v$ has energy less than $\epsilon_{1}$, then we have:

\begin{equation}\label{energy gap inequality}
\int_{B_{1}}|\nabla_{0}u|^{2}-\int_{B_{1}}|\nabla_{0}v|^{2}\geq\frac{1}{2}\int_{B_{1}}|\nabla_{0}u-\nabla_{0}v|^{2}.
\end{equation}
Here we use $\nabla_{0}$ to denote the flat connection of $B_{1}$.
\end{theorem}

\begin{remark}\label{energy gap remark 1}
Although this theorem is formulated when we use the standard metric $ds^{2}_{0}=dx^{2}+dy^{2}$ on $B_{1}$, we can still have inequality (\ref{energy gap inequality}), if we take another metric $ds^{2}$ on $B_{1}$ which is conformal to $ds^{2}_{0}$, since both sides of (\ref{energy gap inequality}) are conformal invariant. Therefore if we take the standard hyperbolic metric $ds^{2}_{-1}$ on a small ball as in the beginning of \S \ref{compactification for mappings}, inequality (\ref{energy gap inequality}) is still true only by changing the flat connection to the connection $\nabla$ of $ds^{2}_{-1}$.
\end{remark}

\begin{remark}\label{energy gap remark 2}
As talked in \cite[\S 4.2]{Z}, we can use the energy gap to control the $W^{1, 2}$-norm difference between a mapping defined on the unit disk with its corresponding energy minimizing harmonic mapping with the same boundary data. This theorem also implies the uniqueness of energy minimizing harmonic maps with energy less than $\epsilon_{1}$ and fixed boundary values \cite[Corollary 3.3]{CM3}.
\end{remark}

Based on this theorem, we have the following result which shows that deforming a mapping locally to the energy minimizing harmonic mapping is a continuous functional. This is a combination of \cite[Corollary 4.1 and 4.2]{Z}, so here we omit the proof.

\begin{corollary}\label{continuity of harmonic replacement}
(\cite[Corollary 3.4]{CM3}\cite[Corollary 4.1 and 4.2]{Z}) Let $\epsilon_{1}$ be given in the previous theorem. Suppose $u\in C^{0}(\overline{B}_{1})\cap W^{1,2}(B_{1})$ with energy $E(u)\leq\epsilon_{1}$, then there exists a unique energy minimizing harmonic map $v\in C^{0}(\overline{B}_{1})\cap W^{1,2}(B_{1})$ with the same boundary value as $u$. Set $\mathcal{M}=\{u\in C^{0}(\overline{B}_{1})\cap W^{1,2}(B_{1}): E(u)\leq\epsilon_{1}\}$.  If we denote $v$ by $H(u)$, then the map $H:\mathcal{M}\rightarrow \mathcal{M}$ is continuous w.r.t. the norm\footnote{Here the norm of $u\in C^{0}(\overline{B}_{1})\cap W^{1,2}(B_{1})$ is given by $\|u\|_{C^{0}(\overline{B}_{1})}+\|u\|_{W^{1,2}(B_{1})}$.} on $C^{0}(\overline{B}_{1})\cap W^{1,2}(B_{1})$.

Suppose that $\{u_{i}\}_{i\in\N}, u$ are defined on a ball $B_{1+\epsilon}$ with energy less than $\epsilon_{1}$, and $\lim_{i\rightarrow\infty}u_{i}= u$ in $C^{0}(\overline{B}_{1+\epsilon})\cap W^{1,2}(B_{1+\epsilon})$. Choose a sequence $r_{i}\rightarrow 1$, and let $w_{i}, w$ be the mappings which coincide with $u_{i}, u$ outside $r_{i}B_{1}$ and $B_{1}$ and are energy minimizing inside $r_{i}B_{1}$ and $B_{1}$ respectively. Then $w_{i}\rightarrow w$ in $C^{0}(\overline{B}_{1+\epsilon})\cap W^{1,2}(B_{1+\epsilon})$.
\end{corollary}

\begin{remark}\label{remark of continuity of harmonic replacement}
If we use geodesic ball $B_{r}$ of geodesic radius $r\leq r_{0}$ on a hyperbolic surface $\Sigma_{0}$ with Poincar\'e metric, all the results of the above lemma hold. This is because that the Poincar\'e metric $ds^{2}_{-1}$ is conformal and uniformly equivalent to the flat metric $ds^{2}_{0}$, so harmonic maps w.r.t. $ds^{2}_{0}$ are also harmonic w.r.t. $ds^{2}_{-1}$, and the $C^{0}$ and $W^{1, 2}$-norms of a fixed map w.r.t. $ds^{2}_{-1}$ are uniformly equivalent to those w.r.t. $ds^{2}_{0}$.
\end{remark}


\subsection{Comparison results of successive harmonic replacements}

Now we will give a comparison result for successive harmonic replacements by adapting  \cite[Lemma 3.11]{CM3} and \cite[Lemma 4.2]{Z}. Fix a mapping $u\in W^{1, 2}(\Sigma_{0}, N)$. We still denote $\mathcal{B}$ as a finite collection of disjoint geodesic balls on $\Sigma_{0}$ as above. Given $\mu\in[0, 1]$, denote $\mu\mathcal{B}$ to be the collection of geodesic balls with the same centers as $\mathcal{B}$, but with geodesic radii $\mu$ timing those corresponding ones of $\mathcal{B}$. Suppose that $u$ has small energy on a collection $\mathcal{B}$. We denote $H(u, \mathcal{B})$ to be the mapping which coincides with $u$ outside $\mathcal{B}$, but are the energy minimizing ones inside $\mathcal{B}$ with the same boundary values as $u$ on $\partial\mathcal{B}$. We call $H$ the harmonic replacement in the following. If $\mathcal{B}_{1}, \mathcal{B}_{2}$ are two such collections, we denote $H(u, \mathcal{B}_{1}, \mathcal{B}_{2})$ to be $H\big(H(u, \mathcal{B}_{1}), \mathcal{B}_{2}\big)$. We have the following energy comparison results for $u$, $H(u, \mathcal{B}_{1})$ and $H(u, \mathcal{B}_{1}, \mathcal{B}_{2})$.

\begin{lemma}\label{comparison}
Fix a Riemann surface $\Sigma_{0}$ (of genus $g\geq 2$) with Poincar\'e metric, and a mapping $u\in C^{0}\cap W^{1,2}(\Sigma_{0}, N)$. Let $\mathcal{B}_{1}$,
$\mathcal{B}_{2}$ be two finite collections of disjoint geodesic balls on $\Sigma_{0}$ with the radius of each ball less than the injective radius of the center of that ball on $\Si_0$ and $r_{0}$ as (\ref{small radius}). If $E(u, \mathcal{B}_{i})\leq\frac{1}{3}\epsilon_{1}$ for $i=1,2$, with $\epsilon_{1}$ given in Theorem \ref{energy gap}, then there exists a constant $k$ depending on $N$, such that:
\begin{equation}\label{comparison inequality1}
E(u)-E[H(u, \mathcal{B}_{1}, \mathcal{B}_{2})]\geq k\bigg(E(u)-E[H(u, \frac{1}{4}\mathcal{B}_{2})]\bigg)^{2},
\end{equation}
and for any $\mu\in[\frac{1}{64}, \frac{1}{4}]$,
\begin{equation}\label{comparison inequality2}
\frac{1}{k}\big(E(u)-E[H(u, \mathcal{B}_{1})]\big)^{\frac{1}{2}}+E(u)-E[H(u, 4\mu\mathcal{B}_{2})]\geq E[H(u, \mathcal{B}_{1})]-E[H(u, \mathcal{B}_{1}, \mu\mathcal{B}_{2})].
\end{equation}
\end{lemma}

\begin{remark}
The proof is similar to that of \cite[Lemma 4.2]{Z}. We will use the Euclidean metric which is conformal to the hyperbolic metric on each of the geodesic balls. Since the inequalities (\ref{comparison inequality1}) and (\ref{comparison inequality2}) are all conformal invariant, the proof in the Euclidean metrics implies that in hyperbolic metrics. By the energy minimizing properties, we can easily get the following inequality:
\begin{equation}\label{comparison inequality3}
E(u)-E[H(u, \mathcal{B}_{1}, \mathcal{B}_{2})]\geq E(u)-E[H(u, \frac{1}{4}\mathcal{B}_{1})].
\end{equation}
This is because that $E[H(u, \mathcal{B}_{1}, \mathcal{B}_{2})]\leq E[H(u, \mathcal{B}_{1})]\leq E[H(u, \frac{1}{4}\mathcal{B}_{1})]$. Combining the above inequalities, we get the comparison for energy of any two successive harmonic replacements by appropriately shrinking the radii.
\end{remark}

We need the following lemma to construct comparison maps. This is a scaling invariant version.

\begin{lemma}\label{construction from boundary}
(\cite[Lemma 3.14]{CM3}) There exists a $\delta>0$ and a large constant $C$ depending on $N$, such that for any $f,g\in C^{0}\cap W^{1,2}(\partial B_{R}, N)$, if $f,g$ are equal at some point on $\partial B_{R}$, and:
\begin{equation}
R\int_{\partial B_{R}}|f^{\prime}-g^{\prime}|^{2}\leq\delta^{2},
\end{equation}
then we can find some $\rho\in(0,\frac{1}{2}R]$, and a mapping $w\in C^{0}\cap W^{1,2}(B_{R}\backslash B_{R-\rho}, N)$ with $w|_{B_{R}}=f$, $w|_{B_{R-\rho}}=g$, which satisfies the estimates:
\begin{equation}
\int_{B_{R}\backslash B_{R-\rho}}|\nabla w|^{2}\leq C\big(R\int_{\partial B_{R}}|f^{\prime}|^{2}+|g^{\prime}|^{2}\big)^{\frac{1}{2}}\big(R\int_{\partial B_{R}}|f^{\prime}-g^{\prime}|^{2}\big)^{\frac{1}{2}}.
\end{equation}
\end{lemma}

~\\
\begin{proof}
(of Lemma \ref{comparison}) Here we will adapt the proof of \cite[Lemma 4.2]{Z}. Since we assume that $E(u, \mathcal{B}_{i})\leq\frac{1}{3}\epsilon_{1}$, we know that $u$ and $H(u, \mathcal{B}_{1})$ have energy less than $\frac{2}{3}\epsilon_{1}$ on $\mathcal{B}_{1}\cup\mathcal{B}_{2}$, so we can use energy gaps to control $W^{1, 2}$-norms difference by Theorem \ref{energy gap}. Denote balls in $\mathcal{B}_{1}$ by $B^{1}_{\alpha}$, and balls in $\mathcal{B}_{2}$ by $B^{2}_{j}$. We prove the two inequalities separately.

\vspace{0.5em}
\textbf{$1^{\circ}$ Inequality (\ref{comparison inequality1}):} We divide the second collection $\mathcal{B}_{2}$ into two sub-collections $\mathcal{B}_{2}=\mathcal{B}_{2+}\cup\mathcal{B}_{2-}$, where $\mathcal{B}_{2+}=\{B^{2}_{j}: \frac{1}{4}B^{2}_{j}\subset B^{1}_{\alpha}\, or\ \frac{1}{4}B^{2}_{j}\cap\mathcal{B}_{1}=\emptyset\ for\ some\ B^{1}_{\alpha}\in\mathcal{B}_{1}\}$ and $\mathcal{B}_{2-}=\mathcal{B}_{2}\setminus\mathcal{B}_{2+}$, and deal with them separately.

For collection $\mathcal{B}_{2+}$, we separate it into another two sub-collections $\{\frac{1}{4}B^{2}_{j}\cap\mathcal{B}_{1}=\emptyset\}$ and $\{\frac{1}{4}B^{2}_{j}\subset B^{1}_{\alpha}\}$. For balls $\frac{1}{4}B^{2}_{j}\cap\mathcal{B}_{1}=\emptyset$, we can use the energy minimizing property of small energy harmonic maps as in Remark \ref{energy gap remark 2}, and similar arguments as \cite[(18)(19)]{Z} to get,
\begin{equation}
\sum_{\{\frac{1}{4}B^{2}_{j}\cap\mathcal{B}_{1}=\emptyset\}}\big(E(u)-E[H(u, \frac{1}{4}B^{2}_{j})]\big)\leq E(u)-E[H(u, \mathcal{B}_{1}, \cup_{\frac{1}{4}B^{2}_{j}\cap\mathcal{B}_{1}=\emptyset}B^{2}_{j})].
\end{equation}

For balls $\frac{1}{4}B^{2}_{j}\subset B^{1}_{\alpha}$, $H(u, \mathcal{B}_{1}, \frac{1}{4}B^{2}_{j})=H(u, \mathcal{B}_{1})$. We denote $u_{1}=H(u, \mathcal{B}_{1})$. Using energy minimizing property of small energy harmonic maps again, and similar arguments as \cite[(20)(21)]{Z}, we have,
\begin{equation}
\begin{split}
\int_{\underset{\frac{1}{4}B^{2}_{j}\subset B^{1}_{\alpha}}{\cup}B^{2}_{j}}|\nabla u|^{2}
&-|\nabla H(u, \frac{1}{4}B^{2}_{j})|^{2} \leq \int_{\underset{\frac{1}{4}B^{2}_{j}\subset B^{1}_{\alpha}}{\cup}B^{2}_{j}}|\nabla u|^{2}-|\nabla H(u, \mathcal{B}_{1}, B^{2}_{j})|^{2}\\
&\leq \int_{\underset{\frac{1}{4}B^{2}_{j}\subset B^{1}_{\alpha}}{\cup}B^{2}_{j}} |\nabla u|^{2}-|\nabla u_{1}|^{2}+\int_{\underset{\frac{1}{4}B^{2}_{j}\subset B^{1}_{\alpha}}{\cup}B^{2}_{j}}|\nabla u_{1}|^{2}-|\nabla H(u, \mathcal{B}_{1}, B^{2}_{j})|^{2}
\end{split}
\end{equation}
The second $``\leq"$ of the above is gotten by adding a term $\int_{\underset{\frac{1}{4}B^{2}_{j}\subset B^{1}_{\alpha}}{\cup}B^{2}_{j}}|\nabla u_{1}|^{2}$ and subtracting a same term after the first $``\leq"$. For the first term, using Theorem \ref{energy gap} and Remark \ref{energy gap remark 1}, we have that $\int_{\underset{\frac{1}{4}B^{2}_{j}\subset B^{1}_{\alpha}}{\cup}B^{2}_{j}}|\nabla u|^{2}-|\nabla u_{1}|^{2}\leq\int_{\underset{\frac{1}{4}B^{2}_{j}\subset B^{1}_{\alpha}}{\cup}B^{2}_{j}}|\nabla u-\nabla u_{1}|^{2}\leq 4\Big(E(u)-E(u_{1})\Big)$. The second term is bounded from above by $E(u_{1})-E[H(u_{1}, \underset{\frac{1}{4}B^{2}_{j}\subset B^{1}_{\alpha}}{\cup}B^{2}_{j})]\leq E(u)-E[H(u, \mathcal{B}_{1}, \underset{\frac{1}{4}B^{2}_{j}\subset B^{1}_{\alpha}}{\cup}B^{2}_{j})]$. So combining the above estimates together, we get inequality,
\begin{equation}
E(u)-E[H(u, \frac{1}{4}\mathcal{B}_{2+})]\leq C\big(E(u)-E[H(u, \mathcal{B}_{1}, \mathcal{B}_{2+})]\big).
\end{equation}

Now let us consider the sub-collection $\mathcal{B}_{2-}$. Here we deal with balls individually. Fix a $B^{2}_{j}\in\mathcal{B}_{2-}$, then $\frac{1}{4}B^{2}_{j}\cap B^{1}_{\alpha}\neq\emptyset$ for some $B^{1}_{\alpha}\in\mathcal{B}_{1}$, but $\frac{1}{4}B^{2}_{j}$ does not belong to any $B^{1}_{\alpha}\in\mathcal{B}_{1}$. Using discussions about small geodesic balls in the beginning of \S \ref{compactification for mappings}, we can identify this $B^{2}_{j}$ with a sub-disk centered at the origin of the Poincar\'e disk, and model it by $(B(0, r^{0}_{B}), \frac{ds^{2}_{0}}{(1-|x|^{2})^{2}})$. Simply denote it by $B_{r^{0}_{B}}$, and denote $u_{1}=H(u, \mathcal{B}_{1})$ as above. Lower subindex here is used to denote the radius of that ball w.r.t. $ds^{2}_{0}$. Now let us construct an auxiliary comparison map. Using Co-area formula, there exists a subset of $[\frac{3}{4}r^{0}_{B}, r^{0}_{B}]$ with measure $\frac{1}{36}r^{0}_{B}$, such that for any $r$ in this subset, we have,
\begin{equation}
\int_{\partial B_{r}}|\nabla_{0}u_{1}-\nabla_{0}u|^{2}\leq\frac{9}{r^{0}_{B}}\int^{r^{0}_{B}}_{\frac{3}{4}r^{0}_{B}}\int_{\partial B_{s}}|\nabla_{0}u_{1}-\nabla_{0}u|^{2}\leq\frac{9}{r}\int_{B_{r^{0}_{B}}}|\nabla_{0}u_{1}-\nabla_{0}u|^{2},
\end{equation}
\begin{equation}
\int_{\partial B_{r}}|\nabla_{0}u_{1}|^{2}+|\nabla_{0}u|^{2}\leq\frac{9}{r^{0}_{B}}\int^{r^{0}_{B}}_{\frac{3}{4}r^{0}_{B}}\int_{\partial B_{s}}|\nabla_{0}u_{1}|^{2}+|\nabla_{0}u|^{2}\leq\frac{9}{r}\int_{B_{r^{0}_{B}}}|\nabla_{0}u_{1}|^{2}+|\nabla_{0}u|^{2},
\end{equation}
where $\nabla_{0}$ is the connection of $ds^{2}_{0}$. By choosing $\epsilon_{1}$ small enough, we can make $r\int_{\partial B_{r}}|\nabla_{0}u_{1}|^{2}+|\nabla_{0}u|^{2}\leq\delta^{2}$ and $r\int_{\partial B_{r}}|\nabla_{0}u_{1}-\nabla_{0}u|^{2}\leq\delta^{2}$ with $\delta$ as in Lemma \ref{construction from boundary}. Since $\frac{1}{4}B_{r^{0}_{B}}\subset B_{\frac{1}{2}r^{0}_{B}}$ as discussed in the beginning of \S \ref{compactification for mappings}, and that $B_{r^{0}_{B}}\in\mathcal{B}_{2-}$, $B_{\frac{1}{2}r^{0}_{B}}$ and hence $B_{r}$ must intersect a ball in $\mathcal{B}_{1}$ but is not contained in any ball of $\mathcal{B}_{1}$, so $u$ and $u_{1}$ must coincide at least one point on $\partial B_{r}$. So by Lemma \ref{construction from boundary}, $\exists \rho\in(0,\frac{1}{2}r]$ and $\exists w\in C^{0}\cap W^{1,2}(B_{r}\backslash B_{r-\rho})$ with $w|_{\partial B_{r}}=u_{1}|_{\partial B_{r}}$, $w|_{\partial B_{r-\rho}}=u|_{\partial B_{r}}$, and:
\begin{equation}\label{construction of w}
\begin{split}
\int_{B_{r}\backslash B_{r-\rho}}|\nabla_{0}w|^{2}
&\leq C\big(r\int_{\partial B_{r}}|\nabla_{0}u_{1}-\nabla_{0}u|^{2}\big)^{\frac{1}{2}}\big(r\int_{\partial B_{r}}|\nabla_{0}u_{1}|^{2}+|\nabla_{0}u|^{2}\big)^{\frac{1}{2}}\\
&\leq C\big(\int_{B_{r^{0}_{B}}}|\nabla_{0}u_{1}-\nabla_{0}u|^{2}\big)^{\frac{1}{2}}\big(\int_{B_{r^{0}_{B}}}|\nabla_{0}u_{1}|^{2}+|\nabla_{0}u|^{2}\big)^{\frac{1}{2}}.
\end{split}
\end{equation}

Now construct comparison map $v$ on $B_{r^{0}_{B}}$ such that:
\begin{displaymath}
v = \left\{ \begin{array}{ll}
u_{1} & \textrm{on $B_{r^{0}_{B}}\backslash B_{r}$}\\
w & \textrm{on $B_{r}\backslash B_{r-\rho}$}\\
H(u, B_{r})(\frac{r}{r-\rho} x) & \textrm{on $B_{r-\rho}$}
\end{array} \right..
\end{displaymath}
In the last equation, we do a rescaling w.r.t. the flat coordinates. Now $E[H(u_{1}, B_{r^{0}_{B}})]\leq E(v)$ on $B_{r^{0}_{B}}$, since $H(u_{1}, B_{r^{0}_{B}})$ is the energy minimizing harmonic map among all maps with the same boundary values. So:
\begin{equation}
\begin{split}
&\int_{B_{r^{0}_{B}}}|\nabla_{0}H(u_{1}, B_{r^{0}_{B}})|^{2}\leq\int_{B_{r^{0}_{B}}}|\nabla_{0}v|^{2}\\
&=\int_{B_{r^{0}_{B}}\backslash B_{r}}|\nabla_{0}u_{1}|^{2}+\int_{B_{r}\backslash B_{r-\rho}}|\nabla_{0}w|^{2}+\int_{B_{r-\rho}}|\nabla_{0}H(u, B_{r})(\frac{r}{r-\rho}\ \cdot)|^{2}\\
&=\int_{B_{r^{0}_{B}}\backslash B_{r}}|\nabla_{0}u_{1}|^{2}+\int_{B_{r}\backslash B_{r-\rho}}|\nabla_{0}w|^{2}+\int_{B_{r}}|\nabla_{0}H(u, B_{r})|^{2}.
\end{split}
\end{equation}
Now since $\frac{1}{4}B_{r^{0}_{B}}\subset B_{\frac{1}{2}r^{0}_{B}}\subset B_{r}$, we have:
\begin{equation}
\begin{split}
&\int_{\frac{1}{4}B_{r^{0}_{B}}}|\nabla_{0}u|^{2}-\int_{\frac{1}{4}B_{r^{0}_{B}}}|\nabla_{0}H(u, \frac{1}{4}B_{r^{0}_{B}})|^{2}\leq\int_{B_{r}}|\nabla_{0}u|^{2}-\int_{B_{r}}|\nabla_{0}H(u, B_{r})|^{2}\\
&\leq\int_{B_{r}}|\nabla_{0}u|^{2}-\int_{B_{r^{0}_{B}}}|\nabla_{0}H(u_{1},
B_{r^{0}_{B}})|^{2}+\int_{B_{r}\backslash B_{r-\rho}}|\nabla_{0}w|^{2}+\int_{B_{r^{0}_{B}}\backslash B_{r}}|\nabla_{0}u_{1}|^{2}\\
&\leq\int_{B_{r^{0}_{B}}}|\nabla_{0}u_{1}|^{2}-\int_{B_{r^{0}_{B}}}|\nabla_{0}H(u_{1},
B_{r^{0}_{B}})|^{2}+\int_{B_{r}\backslash B_{r-\rho}}|\nabla_{0}w|^{2}+\int_{B_{r}}|\nabla_{0}u|^{2}-\int_{B_{r}}|\nabla_{0}u_{1}|^{2}.
\end{split}
\end{equation}
Now we can use the conformal invariance for energy integral to change all the flat connection $\nabla_{0}$ and flat metric $ds^{2}_{0}$ to hyperbolic connection $\nabla$ and hyperbolic metric $ds^{2}_{-1}$. Summing the above inequality on all balls in $\mathcal{B}_{2-}$, and using Theorem \ref{energy gap} and Remark \ref{energy gap remark 1} together with inequality (\ref{construction of w}), we can get the following inequality by similar arguments as those in \cite[(29)(30)]{Z}:
\begin{equation}
E(u)-E[H(u, \frac{1}{4}\mathcal{B}_{2-})]\leq C^{\prime}\big(E(u)-E[H(u, \mathcal{B}_{1}, \mathcal{B}_{2})]\big)^{\frac{1}{2}}.
\end{equation}
Combing inequalities on $\mathcal{B}_{2+}$ and $\mathcal{B}_{2-}$, we get the inequality (\ref{comparison inequality1}).

\vspace{0.5em}
\textbf{$2^{\circ}$ Inequality (\ref{comparison inequality2}):} We divide $\mathcal{B}_{2}$ into two disjoint sub-collections $\mathcal{B}_{2+}$ and $\mathcal{B}_{2-}$, with $\mathcal{B}_{2+}=\{B^{2}_{j}:\mu B^{2}_{j}\subset B^{1}_{\alpha}\ or\ \mu B^{2}_{j}\cap\mathcal{B}_{1}=\emptyset\}$. For collection $\mathcal{B}_{2+}$, similar method also gives:
\begin{equation}
E[H(u, \mathcal{B}_{1})]-E[H(u, \mathcal{B}_{1}, \mu\mathcal{B}_{2+})]\leq E(u)-E[H(u, 4\mu\mathcal{B}_{2+})].
\end{equation}

For subcollection $\mathcal{B}_{2-}$, we use similar proof as above. Here we identify $4\mu B^{2}_{j}$ with a sub-disk centered at the origin of the Poincar\'e disk again, and get an isometric representation $(B_{r^{0}_{B}}, ds^{2}_{-1})$. In the construction of $w$, we change the role of $u$ and $u_{1}$. Let the comparison map be,
\begin{displaymath}
v = \left\{ \begin{array}{ll}
u & \textrm{on $B_{r^{0}_{B}}\backslash B_{r}$}\\
w & \textrm{on $B_{r}\backslash B_{r-\rho}$}\\
H(u_{1}, B_{r})(\frac{r}{r-\rho}\ x) & \textrm{on $B_{r-\rho}$}
\end{array} \right..
\end{displaymath}
We have $\int_{B_{r^{0}_{B}}}|\nabla_{0}H(u, B_{r^{0}_{B}})|^{2}\leq\int_{B_{r^{0}_{B}}}|\nabla_{0}v|^{2}$ by the energy minimizing property. Since we have $\mu B^{2}_{j}=\frac{1}{4}B_{r^{0}_{B}}\subset B_{\frac{1}{2}r^{0}_{B}}$, by argument similar to \cite[(34)(35)[36)]{Z}, we can get,
\begin{equation}
E(u_{1})-E[H(u_{1}, \mu\mathcal{B}_{2-})]\leq E(u)-E[H(u, 4\mu\mathcal{B}_{2-})]+C\big(E(u)-E(u_{1})\big)^{\frac{1}{2}}.
\end{equation}
Combining results on $\mathcal{B}_{2+}$ and $\mathcal{B}_{2-}$, we get inequality (\ref{comparison inequality2}).
\end{proof}


\subsection{Construction of the deformation map}\label{construction of the deformation map}

Let us discuss harmonic replacements on a sweep out- $\big(\gamma(t), \tau(t)\big)\in\tilde{\Omega}$ now. The normalized Fuchsian models of $\tau(t)$ are given by $(\Sigma_{\tau(t)}, \Gamma_{\tau(t)})$, and denote the injective radius of $\Sigma_{\tau(t)}$ by $r_{\tau(t)}$. First, let us point out where to do harmonic replacements. Fix a time parameter $t\in(0, 1)$. Suppose that $B$ is a geodesic ball on $\Sigma_{\tau(t)}$, with radius $r_{B}$ less than the injective radius of the center of $B$ on $\Si_{\tau(t)}$. As discussed in the beginning of \S \ref{compactification for mappings}, we can view $\gamma(t)$ as been defined on the upper half plane $\mathbb{H}$ by lifting up using $\pi_{\tau(t)}: \mathbb{H}\rightarrow\Sigma_{\tau(t)}$. Since $\{\tau(t)\}$ is a compact set in $\mathcal{T}_{g}$, we can always pick one connected component of the pre-images $\pi_{\tau(t)}^{-1}(B)$ inside a fix compact subset $K\subset\mathbb{H}$. Denote that connected component still by $B$, then obviously it has radius $r_{B}$ w.r.t the hyperbolic metric $ds^{2}_{-1}$ of $\mathbb{H}$. Moreover $B$ is a standard ball in $\mathbb{H}$ w.r.t. the flat metric $ds^{2}_{0}$. By the continuity of $\tau(t)$, for parameter $|s-t|\ll 1$, the image of this ball $B$ under $\pi_{\tau(s)}:\mathbb{H}\rightarrow\Si_{\tau(s)}$ is also a geodesic ball with radius less than the injective radius of the center of that ball on $\Si_{\tau(s)}$.
Denoting the image by $B$ again, we will do harmonic replacement simultaneously on $B\subset\Sigma_{\tau(s)}$ for $|s-t|\ll 1$.

When $|s-t|\ll 1$, let us pick up a continuous cutoff function $\mu(s)$, such that $\mu(s)=1$ for $|s-t|\leq \de/2$, and $\mu(s)=0$ for $|s-t|>\delta$ with $\delta>0$ small enough. If we do harmonic replacements for $\gamma(s)$ on balls $\mu(s)B$, Corollary \ref{continuity of harmonic replacement} and Remark \ref{remark of continuity of harmonic replacement} together with the definition of continuity of sweep-outs (\S \ref{notations}) directly imply that we get another continuous sweep-out in $\tilde{\Omega}$. Similarly, we can continuously shrink the radii on balls $\mu(s)B$ where we do harmonic replacements continuously to $0$, so that the new sweep-out can be continuously deformed back to the original one in $\tilde{\Omega}$, which implies that they lie in the same homotopy class by the definition of homotopy equivalence in \S \ref{notations}.

The strategy to construct the deformation map is to first do harmonic replacement on a collection of disjoint geodesic balls where the energy decrease is almost maximal, and then use Lemma \ref{comparison} to get estimate of form (\ref{compactification formula}) for any other harmonic replacements on collection of balls with small energy. For $\sigma\in C^{0}\cap W^{1, 2}(\Sigma_{\tau}, N)$, $\epsilon\in(0, \epsilon_{1}]$, define the maximal possible energy decrease as,
\begin{equation}\label{e_epsilon, sigma}
e_{\epsilon, \sigma}=\sup_{\mathcal{B}}\{E\big(\sigma, \tau\big)-E[H(\sigma, \frac{1}{4}\mathcal{B}), \tau]\},
\end{equation}
where $\mathcal{B}$ are chosen as any finite collection of disjoint geodesic balls on $\Sigma_{\tau}$ with the radius of each ball less than the injective radius of the center of that ball on $\Si_{\tau}$, and $r_0$ as in (\ref{small radius}),
satisfying: $E\big(\sigma, \mathcal{B}\big)\leq\epsilon$. When $\sigma$ is not harmonic, we always have that $e_{\epsilon, \sigma}>0$. Now for a sweep-out $\big(\sigma(t), \tau(t)\big)\in\tilde{\Omega}$, we have the following continuity property similar to \cite[Lemma 3.34]{CM3} and \cite[Lemma 4.4]{Z}.

\begin{lemma}\label{continuity of maximal energy decrease}
$\forall t\in(0,1)$, if $\sigma(t)$ is not harmonic, there exists a neighborhood $I^{t}\subset(0,1)$ of $t$ depending on $t$, $\epsilon$ and $\sigma$, such that $\forall s\in 2 I^{t}$\footnote{$2 I^t$ means the interval with the same center as $I^t$, but twice the length.}.
\begin{equation}
e_{\frac{1}{2}\epsilon, \sigma(s)}\leq 2 e_{\epsilon, \sigma(t)}.
\end{equation}
\end{lemma}
\begin{proof}
Since $e_{\epsilon, \sigma(t)}>0$, the continuity of $\sigma(s)$ implies that that there exists a neighborhood $\tilde{I}^{t}$ of $t$ , such that $\forall s\in 2\tilde{I}^{t}$, and for any finite collection of balls $\mathcal{B}\subset K$, where $K$ is a fixed compact subset of $\mathbb{H}$,
\begin{equation}\label{W1,2 distance of sigma(t) and sigma(s) on balls}
\frac{1}{2}\int_{\mathcal{B}}|\nabla\sigma(s)-\nabla\sigma(t)|^{2}\leq min\big\{\frac{1}{4}e_{\epsilon, \sigma(t)}, \frac{1}{2}\epsilon\big\},
\end{equation}
where we view $\sigma(s)$ as being lifted up to $\mathbb{H}$.

Fix $s\in 2\tilde{I}^{t}$. By Definition \ref{e_epsilon, sigma}, we can pick a finite collection of balls $\mathcal{B}\subset\Sigma_{\tau(s)}$, such that $E(\sigma(s), \mathcal{B})\leq\frac{1}{2}\epsilon$ and $E(\sigma(s))-E[H(\sigma(s), \frac{1}{4}\mathcal{B})]\geq\frac{3}{4}e_{\frac{1}{2}\epsilon, \sigma(s)}$. By taking the compact set $K\subset\mathbb{H}$ large enough, we can always find a connected pre-image in $K$ for each ball in $\mathcal{B}$. Denote those connected pre-image balls by $\mathcal{B}$ again. Then take the image of $\mathcal{B}$ under $\pi_{\tau(t)}: \mathbb{H}\rightarrow\Sigma_{\tau(t)}$, we get another collection of geodesic balls on $\Sigma_{\tau(t)}$, which we still denote by $\mathcal{B}$. So $E(\sigma(t), \mathcal{B})\leq E(\sigma(s), \mathcal{B})+\frac{1}{2}\epsilon\leq\epsilon$ by (\ref{W1,2 distance of sigma(t) and sigma(s) on balls}), hence $E(\sigma(t))-E[H(\sigma(t), \frac{1}{4}\mathcal{B})]\leq e_{\epsilon, \sigma(t)}$ by Definition \ref{e_epsilon, sigma}. So
\begin{equation}
\begin{split}
E\big(\sigma(s)\big) &-E\big[H(\sigma(s), \frac{1}{4}\mathcal{B})\big]\\
                     &\leq |E\big(\sigma(s)\big)-E\big(\sigma(t)\big)|+E\big(\sigma(t)\big)-E\big[H(\sigma(t), \frac{1}{4}\mathcal{B})\big]\\
                     &+|E\big[H(\sigma(t), \frac{1}{4}\mathcal{B})\big]-E\big[H(\sigma(s), \frac{1}{4}\mathcal{B})\big]|.
\end{split}
\end{equation}
Using the continuity of harmonic replacement, i.e. Corollary \ref{continuity of harmonic replacement}, we can possibly shrink the neighborhood $\tilde{I}^{t}$ to a smaller one $I^{t}$, such that $|E\big(\sigma(s)\big)-E\big(\sigma(t)\big)|\leq\frac{1}{4}e_{\epsilon, \sigma(t)}$ and $|E\big[H(\sigma(t), \frac{1}{4}\mathcal{B})\big]-E\big[H(\sigma(s), \frac{1}{4}\mathcal{B})\big]|\leq\frac{1}{4}e_{\epsilon, \sigma(t)}$. Hence $E\big(\sigma(s)\big)-E\big[H(\sigma(s), \frac{1}{4}\mathcal{B})\big]\leq\frac{3}{2}e_{\epsilon, \sigma(t)}$, so $e_{\frac{1}{2}\epsilon, \sigma(s)}\leq 2e_{\epsilon, \sigma(t)}$.
\end{proof}

Next, we will choose families of collections of disjoint geodesic balls corresponding to sweep-outs $\big(\gamma(t), \tau(t)\big)\in\tilde{\Omega}$.

\begin{lemma}\label{selection of balls}
There exist a covering $\{I^{t_{j}}:\ j=1,\cdots,m\}$ for the parameter space $[0, 1]$, and $m$ collections of disjoint geodesic balls $\mathcal{B}_{j}\subset\Sigma_{\tau(t_{j})},\ j=1,\cdots,m$, with the radius of each ball less than the injective radius of the center of that ball on $\Si_{\tau(t_j)}$, and $r_0$ (\ref{small radius}),
together with $m$ continuous functions $r_{j}:[0,1]\rightarrow[0,1]$, $j=1,\cdots,m$, satisfying:

$1^{\circ}$. Each $r_{j}(t)$ is supported in $2I^{t_{j}}$;

$2^{\circ}$. For a fixed $t$, at most two $r_{j}(t)$ are positive, and $E\big(\gamma(t), r_{j}(t)\mathcal{B}_{j}\big)\leq\frac{1}{3}\epsilon_{1}$;

$3^{\circ}$. If $t\in[0,1]$, such that $E\big(\gamma(t), \tau(t)\big)\geq\frac{1}{2}\mathcal{W}$, there exists a $j$, such that $E\big(\gamma(t)\big)-E[H(\gamma(t), \frac{1}{4}r_{j}(t)\mathcal{B}_{j})]\geq\frac{1}{8}e_{\frac{1}{8}\epsilon_{1}, \gamma(t)}$.
\end{lemma}
The proof uses the continuity of $\big(\gamma(t), \tau(t)\big)$ and $e_{\epsilon, \gamma(t)}$ together with a covering argument for the parameter space $[0, 1]$. It is similar to that of \cite[Lemma 3.39]{CM3} and \cite[Lemma 4.5]{Z}, so we omit the proof.

\vspace{1em}
\begin{proof}
(of Lemma \ref{compactification}) The perturbation from $\gamma(t)$ to $\rho(t)$ is done by successive harmonic replacements on the collection of balls given in Lemma \ref{selection of balls}. Denote $\gamma^{0}(t)=\gamma(t)$, and $\gamma^{k}(t)=H\big(\gamma^{k-1}(t), r_{k}(t)\mathcal{B}_{k}\big)$, for $k=1, \cdots, m$. Then $\rho(t)=\gamma^{m}(t)$. Here we can shrink the length of each interval $I^{t_{j}}$, such that the harmonic replacements from $\gamma(t)$ to $\rho(t)$ keep the continuity of $\rho^k(t)$ as discussed in the beginning of this section (\S \ref{construction of the deformation map}). The homotopy equivalence of $\rho(t)$ and $\gamma(t)$ is also a consequence of the discussions there (\S \ref{construction of the deformation map}). Since harmonic replacements decrease energy, we have $E\big(\rho(t)\big)\leq E\big(\gamma(t)\big)$.

Now the property $(*)$ comes from similar argument as in the proof of \cite[ Lemma 4.1]{Z} which originate from the proof of \cite[Theorem 3.1]{CM3}. For $t\in (0, 1)$ such that $E\big(\gamma(t), \tau(t)\big)\geq\frac{1}{2}\mathcal{W}$, we deform $\gamma(t)$ to $\rho(t)$ by at most two harmonic replacements, with the possible middle one denoted by $\gamma^{k}(t)$. Now we focus on the case of two replacements, and the other case is similar and much easier. For any collection $\mathcal{B}$ with $E(\rho(t), \mathcal{B})\leq\frac{1}{12}\epsilon_{1}$, we can assume that both $\gamma(t)$ and $\gamma^{k}(t)$ have energy less than $\frac{1}{8}\epsilon_{1}$ on $\mathcal{B}$, or inequality (\ref{compactification formula}) is trivial. By property $3^{\circ}$ of Lemma \ref{selection of balls}, at least one of the energy decrease from $\gamma(t)$ to $\rho(t)$ is bounded from below by $\frac{1}{8}e_{\frac{1}{8}\epsilon_{1}, \gamma(t)}$. so we have from either inequality (\ref{comparison inequality1}) of Lemma \ref{comparison} or inequality (\ref{comparison inequality3}) that:
\begin{equation}
E\big(\gamma(t)\big)-E\big(\rho(t)\big)\geq k\big(\frac{1}{8}e_{\frac{1}{8}\epsilon_{1}, \gamma(t)}\big)^{2}.
\end{equation}

Now using inequality (\ref{comparison inequality2}) twice for $\mu=\frac{1}{64},\ \frac{1}{16}$, we get:
\begin{equation}\label{energy gap estimate}
\begin{split}
& E\big(\rho(t)\big)-E[H(\rho(t), \frac{1}{64}\mathcal{B})]\\
&\leq E\big(\gamma^{k}(t)\big)-E[H(\gamma^{k}(t), \frac{1}{16}\mathcal{B})]+\frac{1}{k}\big\{E\big(\gamma^{k}(t)\big)-E\big(\rho(t)\big)\big\}^{\frac{1}{2}}\\
&\leq E\big(\gamma(t)\big)-E[H(\gamma(t), \frac{1}{4}\mathcal{B})]+\frac{1}{k}\big\{E\big(\gamma(t)\big)-E\big(\gamma^{k}(t)\big)\big\}^{\frac{1}{2}}\\
&+\frac{1}{k}\big\{E\big(\gamma(t)\big)-E\big(\rho(t)\big)\big\}^{\frac{1}{2}}\\
&\leq e_{\frac{1}{8}\epsilon_{1}, \gamma(t)}+C\big\{E\big(\gamma(t)\big)-E\big(\rho(t)\big)\big\}^{\frac{1}{2}}\leq C\big\{E\big(\gamma(t)\big)-E\big(\rho(t)\big)\big\}^{\frac{1}{2}}.
\end{split}
\end{equation}
By taking $\epsilon_{0}=\frac{1}{12}\epsilon_{1}$ and $\Psi$ the square root function, we can get inequality (\ref{compactification formula}) by using Theorem \ref{energy gap} to change the left hand side of (\ref{energy gap estimate}) to the $W^{1, 2}$-norm difference.
\end{proof}


\section{Convergence results}\label{Convergence results}

Here we talk about the convergence about our deformed sequences $\{\rho_{n}(t), \tau_{n}(t)\}_{n=1}^{\infty}$. In Lemma \ref{compactification}, we need our sequence $\{\gamma_{n}(t), \tau_{n}(t)\}_{n=1}^{\infty}$ to have no non-constant harmonic slices. We can achieve this by an argument similar to \cite[Remark 4.6]{Z}. In fact, we can modify the minimizing sequence $\{\tilde{\gamma}_{n}(t)\}_{n=1}^{\infty}$ such that $\tilde{\gamma}_{n}(t)$ are constant mappings on a small open set on $\Sigma_{0}$, without changing the area too much. By Theorem \ref{conformal parametrization}, $\gamma_{n}(t)$ are gotten from $\tilde{\gamma}_{n}(t)$ by composing with diffeomorphisms $h_{n}(t)$, so $\gamma_{n}(t)$ are also constant mappings on some small open set. By the unique continuation of harmonic maps \cite[Corollary 2.6.1]{J1}, we know that for any parameter $t$, $\gamma_{n}(t)$ could not be harmonic mapping unless it is a constant mapping. So we can apply Lemma \ref{compactification}.

We would also like to preserve the almost conformal property given in Theorem \ref{conformal parametrization} after the deformation given by Lemma \ref{compactification}. Although we could not make sure that $\rho_{n}(t)$ are still almost conformal for every parameter $t$ after the deformation, we can prove similar results for the parameter $t$ with $E\big(\rho_{n}(t_{n}), \tau_{n}(t_{n})\big)$ closed to the min-max critical value $\mathcal{W}$. The proof is almost the same as \cite[Lemma 5.1]{Z}, so we omit the proof here. The result is as following.

\begin{lemma}\label{almost conformal for perturbed sequence}
Given a sequence of parameters $\{t_{n}\}_{n=1}^{\infty}$, such that $E\big(\rho_{n}(t_{n}), \tau_{n}(t_{n})\big)\rightarrow\mathcal{W}$ as $n\rightarrow\infty$, then
\begin{equation}
E\big(\rho_{n}(t_{n}), \tau_{n}(t_{n})\big)-Area\big(\rho_{n}(t_{n})\big)\rightarrow 0,\quad \textrm{as }n\rightarrow\infty.
\end{equation}
\end{lemma}


\subsection{Degeneration of conformal structures}\label{degeneration of conformal structures}

Let us talk about the \textbf{compactification of moduli space} $\mathcal{M}_{g}$. Here we mainly refer to \cite[Appendix B]{IT} and \cite[Chapter IV]{H}\footnote{\cite[\S 4]{Zh} also gives a nice summation in hyperbolic structures.}. In fact, we will use hyperbolic metrics to represent elements in $\mathcal{M}_{g}$ and its compactification. First, let us introduce the representation of the moduli space $\mathcal{M}_{g}$ and Teichm\"uller space $\mathcal{T}_{g}$ by hyperbolic and complex structures. Fix a topological surface $\Sigma_{0}$ of genus $g\geq 2$. Every metric on $\Sigma_{0}$ determines a compatible complex structure $j$ \cite[\S 1.5.1]{IT}. There exists a hyperbolic metric $h$ compatible with $j$. In fact, by the Uniformization Theorem the covering projection $\pi:\mathbb{H}\rightarrow(\Sigma_{0}, j)$ is holomorphic, and the deck transformation group acts isomorphically w.r.t. the hyperbolic metric $ds^{2}_{-1}$. So we can get a hyperbolic metric $h$ on $\Sigma_{0}$ by pushing down $ds^{2}_{-1}$, and this metric is compatible with $j$ since $ds^{2}_{-1}$ is compatible with the standard complex structure on $\mathbb{H}$. Denote such a hyperbolic Riemann surface by a triple $(\Sigma_{0}, h, j)$. Two hyperbolic metrics on $\Sigma_{0}$ are conformal equivalent if and only if they are isomorphic to each other. So we can view $\mathcal{M}_{g}$ as the set of equivalent classes of $(\Sigma_{0}, h, j)$ up to isomorphisms, and $\mathcal{T}_{g}$ as the set of equivalent classes of $(\Sigma_{0}, h, j)$ up to isotrpic isomorphisms.

Now we will introduce the concept of \textbf{Riemann surfaces with nodes}. The precise definition is given in \cite[Appendix B.2]{IT}. A compact connected Hausdorff space $\Sigma^{*}$ is called a \emph{closed Riemann surface of genus $g$ with nodes} if the following conditions hold:
\begin{enumerate}
\vspace{-5pt}
\addtolength{\itemsep}{-0.7em}
\item[(\rom{1})] Every point $p\in\Sigma^{*}$ either has a neighborhood homeomorphic to the unit disk $\{z\in\mathbb{C}: |z|<1\}$ or to the set of one point gluing of two unit disks $\{z_{1}\in\mathbb{C}: |z_{1}|\leq 1\}\cup_{0\rightarrow 0}\{z_{2}\in\mathbb{C}: |z_{2}|\leq 1\}$, 
and in the second case we call $p$ a \emph{node}. These complex coordinates give a complex structure $j$ on $\Sigma^{*}$ minus nodes. Since $\Sigma^{*}$ is compact, there are only finitely many nodes;

\item[(\rom{2})] Let $\Sigma$ be $\Sigma^{*}$ minus nodes, and $\overline{\Sigma}$ the one point compactification of $\Sigma$\footnote{Later on, we will always denote $\Sigma^{*}$ by surface with nodes, $\Sigma$ by surface minus nodes, and $\overline{\Sigma}$ by the one points compactification of $\Sigma$.}. We call $\Sigma$ the \emph{body of $\Sigma^{*}$}. Every connected component $\Sigma_{i}$ of $\Sigma$, which we call it a \emph{part of $\Sigma^{*}$}, is of type $(g_{i}, k_{i})$, which means that $\Sigma_{i}$ is gotten by removing $k_{i}$ distinct points from a Riemann surface of genus $g_{i}$, and we require that $2g_{i}-2+k_{i}>0$. The second condition makes sure that $\Sigma_{i}$ is not homotopic to complex plane or cylinder, which means that $\Sigma_{i}$ has the universal cover $\mathbb{H}$. We call such a part $\Sigma_{i}$ having \emph{signature $(g_{i}, k_{i})$};

\item[(\rom{3})] If $m$ and $k$ denote the numbers of nodes and parts of $\Sigma^{*}$, then the genus $g$ is given by $g=\Sigma_{i=1}^{k}g_{i}+m+1-k$. The last condition tells us that we can get a Riemann surface $\Sigma_{0}$ of genus $g$ from $\Sigma^{*}$ by opening each node.
\end{enumerate}

Two Riemann surfaces with nodes $\Sigma^{*}_{1}$ and $\Sigma^{*}_{2}$ of genus $g$ are said to be \emph{biholomorphically equivalent} if there exists a homeomorphism $f:\Sigma^{*}_{1}\rightarrow\Sigma^{*}_{2}$ preserving nodes, such that $f$ is biholomorphic between parts $(\Sigma_{1})_{i}$ and $(\Sigma_{2})_{i}$ of $\Sigma^{*}_{1}$ and $\Sigma^{*}_{2}$ respectively. If we add the equivalent classes $[\Sigma^{*}]$ of Riemann surfaces with nodes of genus $g$ to the moduli space $\mathcal{M}_{g}$, we get a compactification $\hat{\mathcal{M}}_{g}$ of $\mathcal{M}_{g}$\footnote{We refer to \cite[Appendix B.2 and B.3]{IT} for topology on $\hat{\mathcal{M}}_{g}$ and \cite[Theorem B.1]{IT} for compactness.}.

In fact, we are interested in the convergence $[\Sigma_{n}]\rightarrow[\Sigma^{*}_{\infty}]$ of a sequence of elements in $\mathcal{M}_{g}$ to the boundary of $\hat{\mathcal{M}}_{g}$. We will describe the convergence by representing all the equivalent classes by hyperbolic structures. Now let us first talk about the hyperbolic representation of Riemann surfaces with nodes. Given a Riemann surface with nodes $\Sigma^{*}$ , let $j$ be the complex structure on the body $\Sigma$ of $\Sigma^{*}$. On each part $\Sigma_{i}$, there exists a complete hyperbolic metric $h$ compatible with $j$, with the nodes becoming cusps. So we use $(\Sigma^{*}, h, j)$ to denote a \emph{hyperbolic Riemann surface with nodes}. A triple-connected Riemann surfaces with possibly degenerated boundaries is call a \emph{pair of pants}. Fix a hyperbolic Riemann surface with nodes $(\Sigma^{*}, h, j)$, there exists the pair of pants decomposition\footnote{See \cite[\S 3]{IT} and \cite[Chap IV]{H} for detailed discussion of definitions and properties.}. It means that we can find a largest possible collection of pairwise disjoint, simply closed geodesics $\mathcal{L}=\{\gamma^{i}: i=1\cdots 3g-3\}$ under the hyperbolic metric $h$, with $\gamma^{i}$ possibly degenerating to nodes, such that each connected component of $\Sigma^{*}\setminus\mathcal{L}$ is a pair of pants. Now we give a concept for convergence of a sequence of closed hyperbolic Riemann surfaces of genus $g$ to a hyperbolic Riemann surface with nodes\footnote{For general convergence of a sequence of Riemann surfaces with nodes to a fixed Riemann surface with nodes, see \cite[Page 71]{H}.}.

\begin{definition}\label{convergence of Riemann surfaces}
A sequence $\{(\Sigma_{n}, h_{n}, j_{n})\}$ of closed hyperbolic Riemann surfaces of genus $g$ is said to converge to a hyperbolic Riemann surface with nodes $(\Sigma^{*}_{\infty}, h_{\infty}, j_{\infty})$, if there exists a sequence of finite sets $\mathcal{L}_{n}=\{\gamma_{n}^{i}\}_{i=1}^{k_{n}}\subset\Sigma_{n}$ consisting of pairwise disjoint simply closed geodesics on $(\Sigma_{n}, h_{n})$, with the number of elements $k_{n}$ bounded by $0\leq k_n\leq 3g-3$, and a sequence of continuous mappings $\phi_{n}:\Sigma_{n}\rightarrow\Sigma^{*}_{\infty}$, satisfying the following conditions as $n\rightarrow\infty$:

\vspace{-11pt}
\begin{itemize}
\addtolength{\itemsep}{-0.7em}
\item[$1^{\circ}:$] $\phi_{n}(\gamma_{n}^{i})=p_{i}$, where $p_{i}$ is a node on $\Sigma^{*}_{\infty}$, and the length $l(\gamma_{n}^{i})\rightarrow 0$.
\item[$2^{\circ}:$] $\phi_{n}:\Sigma_{n}\setminus\mathcal{L}_{n}\rightarrow\Sigma_{\infty}$ is a diffeomorphism, where $\Sigma_{\infty}$ is the body of $\Sigma^{*}_{\infty}$.
\item[$3^{\circ}:$] $(\phi_{n})_{*}h_{n}\rightarrow h_{\infty}$ in $C^{\infty}_{loc}(\Sigma_{\infty})$.
\item[$4^{\circ}:$] $(\phi_{n})_{*}j_{n}\rightarrow j_{\infty}$ in $C^{\infty}_{loc}(\Sigma_{\infty})$.
\end{itemize}
\end{definition}

Now using the hyperbolic description of convergence, we can summarize a version of the compactification $\hat{\mathcal{M}}_{g}$ of $\mathcal{M}_{g}$. We refer to \cite[Chap 4, Proposition 5.1]{H} for a proof.

\begin{proposition}\label{compactness of hyperbolic surfaces}
For any sequence $\{(\Sigma_{n}, h_{n}, j_{n})\}_{n=1}^{\infty}$, where each element $(\Sigma_{n}, h_{n}, j_{n})$ represents an equivalent class in $\mathcal{M}_{g}$, there exists a subsequence $\{(\Sigma_{n^{\prime}}, h_{n^{\prime}}, j_{n^{\prime}})\}$ converging to a hyperbolic Riemann surface with nodes $(\Sigma^{*}_{\infty}, h_{\infty}, j_{\infty})$, which represents an equivalent class in $\hat{\mathcal{M}}_{g}$.
\end{proposition}

Besides the convergence results, we also have a detailed description of the geometry near the degenerating geodesics. We refer to \cite[Chap 4, Proposition 4.2]{H} and \cite[Lemma 4.2]{Zh} for the following collar lemma.

\begin{lemma}\label{collar lemma}
For any simply closed geodesic $\gamma$ with length $l(\gamma)=l$ in a hyperbolic surface $(\Sigma, h)$, there exists a collar neighborhood of $\gamma$, which is isomorphic to the following collar region in the hyperbolic plane $\mathbb{H}$:
\begin{equation}\label{collar region}
\mathcal{C}(\gamma)=\big\{z=re^{i\theta}\in\mathbb{H}:\ 1\leq r\leq e^{l},\ \theta_{0}(l)\leq\theta\leq\pi-\theta_{0}(l)\big\},
\end{equation}
with the circles $\{r=1\}$ and $\{r=e^{l}\}$ identified by the isometry $\Gamma_{l}:z\rightarrow e^{l}z$. Here $\theta_{0}(l)=\tan^{-1}\big(\sinh(\frac{l}{2})\big)$, and $\gamma$ is isometric to $\{z=re^{\frac{\pi}{2}i}\in i\mathbb{R}:\ 1\leq r\leq e^{l}\}$.
\end{lemma}

\begin{remark}\label{remark of collar lemma}
In fact, this result follows from the proof of \cite[Chap 4, Lemma 1.6]{H}. They consider half of the collar, and they show that the collar region should be part of annuli $\{re^{i\theta}: \theta_{0}\leq\theta\leq\frac{\pi}{2}, 1\leq r\leq y\}$. Instead of using polar coordinates $\{r, \theta\}$, they use the length of boundary of the region $\{re^{i\theta}: \theta_{0}\leq\theta\leq\frac{\pi}{2}, r=1\}$ as parameter. It is easy to change back to polar coordinates and get our formulation above.
\end{remark}

As stated in \cite{Z}, we can give a explicit metric on the collar region by a conformal change of coordinates. Now, we can view the parameters $r$ and $\theta$ in (\ref{collar region}) as azimuthal and vertical coordinates for a cylinder respectively. Under the following transformation:
$$re^{i\theta}\rightarrow(t, \phi)=\big(\frac{2\pi}{l}\theta, \frac{2\pi}{l}\log(r)\big),$$
where $l$ is the length of the center geodesic, the collar region $\mathcal{C}(\gamma)$ is changed to a cylinder
\begin{equation}
C=\big\{(t, \phi):\ \frac{2\pi}{l}\theta_{0}\leq t\leq\frac{2\pi}{l}(\pi-\theta_{0}), 0\leq\phi\leq2\pi\big\},
\end{equation}
and the hyperbolic metric $ds_{-1}^{2}=\frac{|dz|^{2}}{(Imz)^{2}}$ is expressed as $ds_{-1}^{2}=(\frac{l}{2\pi\sin(\frac{l}{2\pi}t)})^{2}(dt^{2}+d\phi^{2})$, which is conformal to the standard cylindrical metric $ds^{2}=dt^{2}+d\phi^{2}$. We can see that if the geodesic $\gamma$ shrink to a point, a conformally infinitely long cylinder will appear.


\subsection{Convergence}

Before talking about bubble tree convergence of the sequence $\big\{\rho_{n}(t), \tau_{n}(t)\big\}_{n=1}^{\infty}$ gotten by the previous section, let us first clarify the concepts of convergence for a sequence $\{\tau_{n}\}_{n=1}^{\infty}\subset\mathcal{T}_{g}$. Since the area and energy functionals are both conformally invariant, we can choose good representatives in the conformal classes of $\{\tau_{n}\}_{n=1}^{\infty}$. Or in another word, we world like to project $\mathcal{T}_{g}$ to $\mathcal{M}_{g}$, and use the compactification $\hat{\mathcal{M}}_{g}$ of $\mathcal{M}_{g}$ to discuss the convergence of $\{\tau_{n}\}_{n=1}^{\infty}$. Here we use hyperbolic representatives as talked above. We say \emph{$\{\tau_{n}\}_{n=1}^{\infty}$ converge to $\tau_{\infty}$ in $\hat{\mathcal{M}}_{g}$}, if we can find hyperbolic representatives $(\Sigma_{n}, h_{n}, j_{n})\in\tau_{n}$ and $(\Sigma^{*}_{\infty}, h_{\infty}, j_{\infty})\in\tau_{\infty}$, such that $(\Sigma_{n}, h_{n}, j_{n})$ converge to $(\Sigma^{*}_{\infty}, h_{\infty}, j_{\infty})$ in the sense of Definition \ref{convergence of Riemann surfaces}. In another word, if we denote $[\tau]$ to be the projection of $\tau$ to $\mathcal{M}_{g}$, \emph{the convergence of $\{\tau_{n}\}$ to $\tau_{\infty}$ means that $[\tau_{n}]$ converge to $[\tau_{\infty}]$ in $\hat{\mathcal{M}}_{g}$.} Now we can state the following theorem.

\begin{theorem}\label{convergence}
(Theorem \ref{main theorem 2}) Let $\big\{(\rho_{n}(t), \tau_{n}(t))\big\}_{n=1}^{\infty}$ be the sequence gotten by the perturbation from $\big\{(\gamma_{n}(t), \tau_{n}(t))\big\}_{n=1}^{\infty}$ by Lemma \ref{compactification}\footnote{See the discussion in the beginning of \S \ref{Convergence results} on how to achieve the no non-constant harmonic slice condition.}, then all min-max sequences $\{(\rho_{n}(t_{n}), \tau_{n}(t_{n}))\}_{n=1}^{\infty}$ with $E\big(\rho_{n}(t_{n}), \tau_{n}(t_{n})\big)\rightarrow\mathcal{W}_{E}$, satisfy:
\begin{itemize}
\vspace{-5pt}
\addtolength{\itemsep}{-0.7em}

\item[(*)] For any finite collection of disjoint geodesic balls $\underset{i}{\cup}B_{i}$ on $\Sigma_{\tau_{n}(t_{n})}$ with radii bounded as in Lemma \ref{compactification}, such that $E\big(\rho_{n}(t_{n}), \underset{i}{\cup}B_{i}\big)\leq\epsilon_{0}$, let $v$ be the harmonic replacement of $\rho_{n}(t_{n})$ on $\frac{1}{64}\underset{i}{\cup}B_{i}$, then we have:
\begin{equation}
\int_{\frac{1}{64}\underset{i}{\cup}\mathcal{B}_{i}}|\nabla\rho_{n}(t_{n})-\nabla v|^{2}\rightarrow 0
\end{equation}
\end{itemize}
By Proposition \ref{compactness of hyperbolic surfaces}, a subsequence of $\{\tau_{n}(t_{n})\}_{n=1}^{\infty}$ converge to some $\tau_{\infty}$ in $\hat{\mathcal{M}}_{g}$, which is achieved by the convergence of a sequence of hyperbolic Riemann surfaces $(\Sigma_{n}, h_{n}, j_{n})\in\tau_{n}(t_{n})$ to $(\Sigma^{*}_{\infty}, h_{\infty}, j_{\infty})\in\tau_{\infty}$ as in Definition \ref{convergence of Riemann surfaces}. If we denote the one point compactification of $\Sigma_{\infty}$ by $\overline{\Sigma}_{\infty}$, and $j_{\infty}$ the extended complex structure, then there exist a conformal harmonic map $u_{0}:\big(\overline{\Sigma}_{\infty}, j_{\infty}\big)\rightarrow N$ and possibly some harmonic spheres $\{u_{i}:S^{2}\rightarrow N|\ i=1, \cdots, l\}$, such that $\big(\rho_{n}(t_{n}), (\Sigma_{n}, h_{n}, j_{n})\big)$ bubble tree converge\footnote{See \S \ref{further discussion}. We refer to \cite{SU1, SU2, Pa} and \cite[Appendix B.6]{CM3} for more details about bubble tree convergence.} to $\big(u_{0}, u_{1}, \ldots, u_{l}\big)$, with energy identity:
\begin{equation}\label{energy identity}
\underset{n\rightarrow\infty}{\lim}E\big(\rho_{n}(t_{n}), j_{n}\big)=E(u_{0}, j_{\infty})+\underset{i}{\sum}E(u_{i})
\end{equation}
\end{theorem}

\begin{remark}
In fact, property $(*)$ in the above theorem is scaling invariant, so we can apply the Sacks-Uhlenbeck's bubble tree convergence theory to $\{\rho_{n}(t_{n})\}$. In fact, the left hand side of (\ref{energy identity}) is the min-max critical value $\mathcal{W}$, and the right side is the sum of areas since $\big(u_{0}, u_{1}, \ldots, u_{l}\big)$ are all conformal, so we get the conclusion that the min-max critical value is achieved by the area of a set of minimal surfaces.
\end{remark}

The proof is divided into several steps in the following sections.


\subsubsection{Convergence on domains}\label{convergence on domains}

First we summarize some known facts of convergence of almost harmonic maps defined on a sequence of converging domains. Suppose that $\{(\Omega_{n}, h_{n}, j_{n})\}_{n=1}^{\infty}$ is a sequence of two dimensional domains with metrics $h_{n}$ and compatible complex structures $j_{n}$. We assume that $(\Omega_{n}, h_{n}, j_{n})\rightarrow(\Omega_{\infty}, h_{\infty}, j_{\infty})$ in the following sense. For $n$ large enough, there exist a sequence of diffeomorphisms $\phi_{n}:\Omega_{\infty}\rightarrow\Omega_{n}$, such that the pull-back metrics and complex structures converge, i.e. $(\phi_{n})^{*}h_{n}\rightarrow h_{\infty}$ and $(\phi_{n})^{*}j_{n}\rightarrow j_{\infty}$ in $C^{3}$ on any compact subsets of $\Omega_{\infty}$. Let $\{u_{n}:(\Omega_{n}, h_{n}, j_{n})\rightarrow N\}_{n=1}^{\infty}$ be a sequence of $W^{1, 2}$ almost harmonic maps satisfying the following condition:
\begin{itemize}
\vspace{-5pt}
\addtolength{\itemsep}{-0.7em}

\item[$(*1)$] For any geodesic small ball $B\in\Omega_{n}$ with radius less than the the injective radius of the center of the ball on $(\Om_n, h_n)$,
if $E(u_{n}, B)<\epsilon_{0}$ with $\epsilon_{0}$\footnote{In order to apply Sacks-Uhlenbeck's bubble tree convergence theory, we can pick $\epsilon_{0}<\epsilon_{SU}$, where $\epsilon_{SU}$ is a small constant depending only on the ambient manifold $N$ given in \cite[3.2]{SU1}.} given by Lemma \ref{compactification}, denote $v$ to be the harmonic replacement of $u_{n}$ on $\frac{1}{64}B$, then:
$$\int_{\frac{1}{64}B}|\nabla u_{n}-\nabla v|^{2}\leq\delta(n)\rightarrow 0.$$
\end{itemize}

\begin{lemma}\label{first bubble converge}
For a sequence $\{u_{n}:(\Omega_{n}, h_{n}, j_{n})\rightarrow N\}_{n=1}^{\infty}$ as above with $E(u_{n}, j_{n})\leq E_{0}<\infty$, there exist finitely many points $\{x_{1},\cdots,x_{k}\}\subset\Omega_{\infty}$, a subsequence $\{n^{\prime}\}$ and a harmonic mapping $u_{\infty}\in W^{1, 2}(\Omega\setminus\{x_{1},\cdots,x_{k}\}, N)$, such that for any compact subset $K\subset\Omega_{\infty}\setminus\{x_{1},\cdots,x_{k}\}$, the subsequence $u_{n^{\prime}}:(\phi_{n^{\prime}}(K)\subset\Omega_{n^{\prime}}, h_{n^{\prime}}, j_{n^{\prime}})\rightarrow N$ converge to $u_{\infty}$ in $W^{1, 2}$.
\end{lemma}

\begin{remark}
The convergence of $u_{n^{\prime}}$ to $u_{\infty}$ can be understood as the convergence after pulling $u_{n^{\prime}}$ back to $\Omega_{\infty}$ by $\phi_{n^{\prime}}$. We call points $\{x_{1},\cdots,x_{k}\}$ the energy concentration points. The proof of results similar to the above lemma is given in \cite{SU1, SU2}, \cite[Appendix B.2]{CM3} and the proof of \cite[Theorem 5.1]{Z}. In fact, step 1 of the proof of \cite[Theorem 5.1]{Z} almost directly gives the proof of the above lemma, so we omit it. By the Removable Singularity Theorem \cite[Theorem 3.6]{SU1}, we can extend $u_{\infty}$ to a harmonic map on $\Omega_{\infty}$.
\end{remark}


\subsubsection{Convergence on cylinders}\label{convergence on cylinders}

Now based on the above lemma, the next step to study the convergence of $\{(\rho_{n}, \tau_{n})\}_{n=1}^{\infty}$ is to do rescaling near energy concentration points, and then consider regions near degenerating geodesics. In both of the cases which we will discuss in detail later, we need to consider almost harmonic maps on long cylinders. We use $\mathcal{C}_{t^{1},t^{2}}=\{(t, \theta)\in\mathbb{R}\times S^{1}:\ t^{1}\leq t\leq t^{2}, \theta\in[0, 2\pi)\}$ to denote a cylinder with length parameter between $t^{1}$ and $t^{2}$, and \emph{$h$ a metric on $\mathcal{C}_{t^{1},t^{2}}$ conformal to the standard metric $ds^{2}=dt^{2}+d\theta^{2}$}. We denote $S_{t^{0}}=\{(t, \theta): t=t^{0}, \theta\in[0, 2\pi)\}$ to be a slice of $\mathcal{C}_{t^{1},t^{2}}$. We say a sequence of cylinders $\{(\mathcal{C}_{t^{1}_{n}, t^{2}_{n}}, h_{n}):\ 1\leq n<\infty\}$ converge to $(\mathcal{C}_{\infty}=\mathbb{R}\times S^{1}, ds^{2}=dt^{2}+d\theta^{2})$, if when we identify all the cylinders by the center slices $S_{t^{0}_{n}}$ with $t^{0}_{n}=\frac{1}{2}(t^{1}_{n}+t^{2}_{n})$, the metrics $h_{n}$ converges in $C^{3}$ to $ds^{2}$ on any compact subsets of $\mathcal{C}_{\infty}$, i.e. when we choose $\phi_{n}: \mathcal{C}_{t^{1}_{n}, t^{2}_{n}}\rightarrow\mathcal{C}_{\infty}$, such that $\phi_{n}(t, \theta)=(t-t^{0}_{n}, \theta)$, then $(\phi_{n})_{*}h_{n}\rightarrow ds^{2}$ in $C^{3}(K)$ for any compact subset $K\subset\mathcal{C}_{\infty}$. Consider a sequence of almost harmonic maps defined on a sequence of converging cylinders $\{u_{n}:(\mathcal{C}_{t^{1}_{n}, t^{2}_{n}}, h_{n})\rightarrow N|\ n=1,\cdots,\infty\}$ satisfying property $(*1)$ in \S \ref{convergence on domains}. By Lemma \ref{first bubble converge}, they sub-converge to a harmonic map on $\mathcal{C}_{\infty}$. Before discussing further results, we need to introduce another type of almost harmonic maps and a corresponding energy estimate.

\begin{definition}\label{almost harmonic maps}
For $\nu >0$, we call $u\in W^{1,2}\big((\mathcal{C}_{r_{1}, r_{2}}, h), N\big)$ a \textbf{$\nu$-almost harmonic map} (see \cite[Definition B.27]{CM3}) if for any finite collection of disjoint geodesic balls $\mathcal{B}$ in $(\mathcal{C}_{r_{1}, r_{2}}, h)$ with the radius of each ball bounded by the injective radius of the center of that ball on $(\mathcal{C}_{r_{1}, r_{2}}, h)$, 
there is an energy minimizing map $v:\frac{1}{64}\mathcal{B}\rightarrow N$ with the same boundary value as $u$ such that:
\begin{equation}
\int_{\frac{1}{64}\mathcal{B}}|\nabla u-\nabla v|^{2}\leq\nu\int_{\mathcal{C}_{r_{1}, r_{2}}}|\nabla u|^{2}.
\end{equation}
\end{definition}
This definition traces back to \cite[Definition B.27]{CM3}, but we modify it here to be adapted to our setting. Now a proof similar to that of \cite[Proposition B.29]{CM3} gives a similar estimate as follows.

\begin{proposition}\label{almost harmonic maps on necks}
For any $\delta>0$, there exist small constants $\nu>0$ (depending on $h$, $\delta$ and $N$), $\epsilon_{2}>0$ and a large constant $l\geq 1$ (depending on $\delta$ and $N$), such that for any positive integer $m$, if $u$ is a $\nu$-almost harmonic map defined on $(\mathcal{C}_{-(m+3)l, 3l}, h)$ with $E(u)\leq\epsilon_{2}$\footnote{We can let $\epsilon_{2}<\epsilon_{SU}$ again as above.}, then:
\begin{equation}
\int_{\mathcal{C}_{-ml, 0}}|u_{\theta}|^{2}\leq 7\delta\int_{\mathcal{C}_{-(m+3)l, 3l}}|\nabla u|^{2}.
\end{equation}
Here $u_{\theta}$ means the differentiation w.r.t. $\theta$.
\end{proposition}

Now we would like to give a more precise description of the convergence on cylinders.

\begin{lemma}\label{bubble converge on cylinders}
In the convergence of $u_{n}:(\mathcal{C}_{t^{1}_{n}, t^{2}_{n}}, h_{n})\rightarrow N$ as discussed above, if $E(u_{n})\leq\epsilon_{2}$ with $\epsilon_{2}$ given in Proposition \ref{almost harmonic maps on necks}, then either $\liminf_{n\rightarrow\infty}E(u_{n})=0$, or $u_{n}$ must be uniformly un-conformal for $n$ large enough in the following sense, i.e. there exists a small number $\delta_{0}>0$, such that:
\begin{equation}
E(u_{n})-Area(u_{n})\geq\delta_{0}.
\end{equation}
Furthermore, if $\{u_{n}\}$ are almost conformal, i.e. $\lim_{n\rightarrow\infty}\big(E(u_n)-Area(u_n)\big)-0$, and satify that $\liminf_{n\rightarrow\infty}E(u_{n})\geq\epsilon_{2}$, then there exists a large fixed number $L>0$, such that $E(\rho_{n}, \mathcal{C}_{r_{n}^{0}-L, r_{n}^{0}+L})\geq\epsilon_{2}$, i.e. the energy must concentrate on some finite part of the cylinders.
\end{lemma}

\begin{remark}
This is a summarization of the results proved in step 5 of the proof of \cite[Theorem 5.1]{Z}. In fact, if $E(u_{n})\leq\epsilon_{2}$ and $\liminf_{n\rightarrow\infty}E(u_{n})>0$, it is easy to show that $u_{n}$ is $\mu$-almost harmonic as in Definition \ref{almost harmonic maps} for $\mu$ small enough when $n$ is large enough. If we apply the estimate in Proposition \ref{almost harmonic maps on necks}, we get an upper bound for $\int_{\mathcal{C}_{-ml, 0}}|(u_{n})_{\theta}|^{2}$. Then by computing the difference between the energy and area of $u_{n}$ as in \cite[(55)]{Z}, we will get the lower bound for $E(u_{n})-Area(u_{n})$. In the second case, we use contradiction argument. We will go back to the first case to get a sequence of almost harmonic mappings on long cylinders with energy bounded from above by $\epsilon_{2}$ and away from $0$, which will lead to a contradiction to the almost conformal property. We omit the detailed proof here and refer that to \cite{Z}.
\end{remark}


\subsubsection{Proof of Theorem \ref{convergence}}

Now we use the results summarized above to show the bubble tree convergence and the energy identity (\ref{energy identity}) of Theorem \ref{convergence}. Let us denote $\rho_{n}=\rho_{n}(t_{n})$, and $\tau_{n}=\tau_{n}(t_{n})$ in the following.

\textbf{Step 1: bubble tree convergence on domain surfaces.} In the convergence of $(\Sigma_{n}, h_{n}, j_{n})\in\tau_{n}$ to $(\Sigma^{*}_{\infty}, h_{\infty}, j_{\infty})\in\tau_{\infty}$, let us denote $\mathcal{L}_{n}$ to be the sets of geodesics and $\phi_{n}:\Sigma_{n}\rightarrow\Sigma^{*}_{\infty}$ the continuous mappings as in Definition \ref{convergence of Riemann surfaces}. Now let us consider the sequence of almost harmonic maps $\{\rho_{n}:(\Sigma_{n}\setminus\mathcal{L}_{n}, h_{n}, j_{n})\rightarrow N\}_{n=1}^{\infty}$ satisfying property $(*)$ in Theorem \ref{convergence}. By Lemma \ref{first bubble converge}, there exists a finite set of energy concentration points $\{x_{1},\cdots,x_{l}\}$ on the body $\Sigma_{\infty}$ of $\Sigma^{*}_{\infty}$, and a subsequence which we still denote by $\rho_{n}$, that converge to a harmonic map $u_{0}:\Sigma_{\infty}\rightarrow N$ in $W^{1, 2}$ on any compact subsets of $\Sigma_{n}\setminus(\mathcal{L}_{n}\cup\phi_{n}^{-1}\{x_{1},\cdots,x_{l}\})$. Denote $x_{n, i}=\phi_{n}^{-1}(x_{i})$. Near each energy concentration point $x_{n, i}$, let $r_{n,i}$ be the smallest radii such that $E(\rho_{n}, B_{x_{n, i}, r_{n,i}})=\epsilon_{0}$ with $\epsilon_{0}$ as in condition $(*1)$ of \S \ref{convergence on domains}, where $B_{x_{n, i}, r}$ denotes the hyperbolic geodesic balls centered at $x_{n, i}$ with radii $r$ on $\Sigma_{n}$. View $B_{x_{n, i}, r_{n, i}}$ as a ball on the Poincar\'e disk $(D, ds^{2}_{-1})$ centered at the origin $0$, and use the coordinates there. Now rescale $B_{x_{n, i}, r_{n, i}}$ to $B_{0, 1}\subset\mathbb{C}$ by $x\rightarrow x/r_{n, i}^{0}$, where $r_{n, i}^{0}$ is the Euclidean radius of $B_{x_{n, i}, r_{n, i}}$ measured w.r.t. the Euclidean metric on $(D, ds^{2}_{0})$. In fact, $r_{n, i}$ and $r_{n, i}^{0}$ are almost the same when $r_{n, i}\rightarrow 0$ as $n\rightarrow\infty$. Then rescale the hyperbolic metric $ds^{2}_{-1}$ to be $ds^{2}_{n}=\frac{|dz|^{2}}{1-|r_{n, i}^{0}z|^{2}}$, which converge to the flat metric on any compact subsets of $\mathbb{C}$. Let $u_{n, i}(x)=\rho_{n}(r_{n, i}^{0}x)$. Since properties $(*)$ and $(*1)$ are scaling invariant, the sequence $\big\{\big(u_{n, i}, (B_{r/r_{n, i}^{0}}, ds^{2}_{n})\big)\big\}_{n=1}^{\infty}$ satisfy the requirement of Lemma \ref{first bubble converge} again for some fixed small radius $r$. So a subsequence of $\{u_{n, i}\}_{n=1}^{\infty}$ converge in $W^{1, 2}$ to a harmonic map $u_{\infty, i}$ defined on $\mathbb{C}$ in the sense of Lemma \ref{first bubble converge} again. We can repeat such processes near energy concentration points step by step. An important observation is that $u_{\infty, i}:\mathbb{C}\rightarrow N$ is an nontrivial harmonic map, since the energy of $u_{n, i}$ over $B(0, 1)$ is $\epsilon_{1}$ by the conformal invariance of energy and our choice of the bubbling region $B_{x_{n, i}, r_{n, i}}$. Then $u_{\infty, i}$ extends to a harmonic map on the sphere, whose energy is bounded below by $\epsilon_{SU}$ \cite[Theorem 3.3]{SU1}. We call all such harmonic spheres bubbles. So for each step, the total energy is decreased by some fixed amount, hence it must stop in finitely many steps.

\vspace{0.5em}
\textbf{Step 2: bubble tree convergence on necks and collars.} To prove the energy identity (\ref{energy identity}), we need to study the behavior of the limit process on some small annuli and collar neighborhoods of degenerating geodesics. Near an energy concentration point, if we compare the energy limit $\lim_{n\rightarrow\infty}E(\rho_{n}, B(x_{i}, r))$ with the sum of the limit energy $E(u_{0}, B(x_{i}, r))$ and the bubble energy $\lim_{n\rightarrow\infty}E(u_{n, i}, B_{r/r_{n, i}^{0}})$, we need to count the neck part, which is given by
$$\lim_{r\rightarrow 0, R\rightarrow\infty}\lim_{n\rightarrow\infty}E\big(\rho_{n}, B(x_{i}, r)\backslash B(x_{i}, r_{n, i}^{0}R)\big).$$
Here we refer to the step 4 in the proof  of \cite[Theorem 5.1]{Z} for details. Denote the annuli by $A(x_{i}, r, r_{n, i}^{0}R)=B(x_{i}, r)\setminus B(x_{i}, r_{n, i}^{0}R)$, and we call them necks. Under the change of coordinates $(r, \theta)\rightarrow(t, \theta)=(\log r, \theta)$, the annuli are changed to long cylinders $\mathcal{C}_{r^{1}_{n}, r^{2}_{n}}$, with $r^{1}_{n}=\ln(r_{n, i}R)$, $r^{2}_{n}=\ln(r)$, and the hyperbolic metrics are $ds^{2}_{-1}=\frac{e^{2t}}{1-e^{2t}}(dt^{2}+d\theta^{2})$. When we rescale the metrics such that the center slice $S_{t_{n}^{0}}$ has length $2\pi$, it is easy to see that the metrics converge to the flat metric on any compact subset of the infinite long cylinder $\mathbb{R}\times S^{1}$. Since property $(*)$ is invariant under scaling, we go back to the setting of \S \ref{convergence on cylinders}. We will continue studying the convergence in this case after we introduce the behavior near degenerating geodesics.

Now let us see the behavior near degenerating geodesics $\gamma_{n}^{i}\in\mathcal{L}_{n}$. Similar arguments as in the case of necks show that if we want to recover all the energy of $\rho_{n}$ on $\Sigma_{n}$ from the limit $u_{0}$ and all the bubbles $u_{i}:S^2\rightarrow N$, we need to consider the amount of energy on the collar neighborhoods $\mathcal{C}(\gamma_{n}^{i})$ given by Lemma \ref{collar lemma}. As in (\ref{collar region}), we use $(r, \theta)$ as parameters for the cylinder, and denote $\mathcal{C}(\gamma_{n}^{i}, \theta_{0})$ to be the sub-collar with $\theta_{0}\leq\theta\leq\pi-\theta_{0}$. In fact, as $l_{n}=l(\gamma_{n}^{i})\rightarrow 0$, we need to take care of the limit $\lim_{\theta_{0}\rightarrow\frac{\pi}{2}}\lim_{n\rightarrow\infty}E(\rho_{n}, \mathcal{C}(\gamma_{n}^{i}, \theta_{0}))$. Using the change of coordinates given in Remark \ref{remark of collar lemma}, those collars can be viewed as a sequence of cylinders $\mathcal{C}_{r_{n}^{1}, r_{n}^{2}}$ with $r_{n}^{1}=\frac{2\pi}{l_{n}}\theta_{0}$, $r_{n}^{2}=\frac{2\pi}{l_{n}}(\pi-\theta_{0})$. If we rescale the hyperbolic metrics $ds_{-1}^{2}=(\frac{l_{n}}{2\pi\sin(\frac{l_{n}}{2\pi}t)})^{2}(dt^{2}+d\phi^{2})$ on $\mathcal{C}_{r_{n}^{1}, r_{n}^{2}}$ such that the center slice $S_{(\frac{2\pi}{l_{n}})\frac{\pi}{2}}$ has length $2\pi$, it is easy to see that those metrics converge to the flat metric on any compact subset of $\mathbb{R}\times S^{1}$ , which goes back to the setting for the \S \ref{convergence on cylinders} again by the conformal invariance of property $(*)$.

Summarizing the above two paragraphs, we need to study the case of a sequence of almost harmonic maps defining on cylinders approximating the infinite long standard cylinder. If $\liminf_{n\rightarrow\infty}E\big(\rho_{n}, (\mathcal{C}_{r_{n}^{1}, r_{n}^{2}}, ds_{-1}^{2})\big)=0$, then we can discard this part in the energy identity (\ref{energy identity}), or since the sequence of maps are almost conformal by Lemma \ref{almost conformal for perturbed sequence}, we have that $\liminf_{n}E\big(\rho_{n}, (\mathcal{C}_{r_{n}^{1}, r_{n}^{2}}, ds_{-1}^{2})\big)\geq\epsilon_{2}$ by Lemma \ref{bubble converge on cylinders}. Then there exists a large fixed number $L>0$, such that $E(\rho_{n}, \mathcal{C}_{r_{n}^{0}-L, r_{n}^{0}+L})\geq\epsilon_{2}$ by the second part of Lemma \ref{bubble converge on cylinders}. Now $\big(\rho_{n}, (\mathcal{C}_{r_{n}^{1}, r_{n}^{2}}, ds_{-1}^{2})\big)$ converge in $W^{1, 2}$ to a harmonic map $u_{\infty}:\mathbb{R}\times S^{2}\rightarrow N$ on any compact subsets of $\mathbb{R}\times S^{2}$ minus possibly finite many energy concentration points by Lemma \ref{first bubble converge}. We can repeat the above steps near energy concentration points again. Now in order to count all the energy, we need to consider sub-cylinders $\mathcal{C}_{t_{n}-L_{n}, t_{n}+L_{n}}\subset\mathcal{C}_{r_{n}^{1}, r_{n}^{2}}$ with $|t_{n}-t_{n}^{0}|\rightarrow\infty$ and $L_{n}\rightarrow\infty$. We need to show that $\lim_{n\rightarrow\infty}E\big(\rho_{n}, \mathcal{C}_{t_{n}-L_{n}, t_{n}+L_{n}}\big)$ is counted by some bubble maps. In fact, when we rescale the metrics such that the center slice $S_{t_{n}}$ of $\mathcal{C}_{t_{n}-L_{n}, t_{n}+L_{n}}$ has length $2\pi$, the sequence of cylinders will converge to $\mathbb{R}\times S^{1}$ again as above. So we can repeat the steps again.

We can see that no energy loss will happen since once there are energy concentrated on long cylinders, they must be counted in the next bubbling step. We know that either $u_{\infty}:\mathbb{R}\times S^{1}\rightarrow N$ is nontrivial, which can be extended to a harmonic map on $S^{2}$ by the Removable Singularity Theorem \cite[Theorem 3.6]{SU1}, since $S^{2}$ is conformal to $\mathbb{R}\times S^{1}$, or some of the bubble maps near energy concentration points are nontrivial since $E(\rho_{n}, \mathcal{C}_{r_{n}^{0}-L, r_{n}^{0}+L})\geq\epsilon_{2}$. So each of such steps also takes away a fixed amount of energy, so we must stop in finite many steps. All such steps form the convergence in Theorem \ref{convergence}. Count all the energy of those finitely many bubble maps, which are harmonic maps on spheres, we will get the energy identity (\ref{energy identity}). So we finish the proof.


\parindent 0ex
Department of Mathematics, Stanford University\\
Stanford, CA 94305\\
E-mail: xzhou08@math.stanford.edu\\

{\em Current address}: Massachusetts Institute of Technology\\
E-mail: xinzhou@math.mit.edu

\end{document}